\documentclass[11pt]{amsart}

\makeatletter
\def\@seccntformat#1{%
  \protect\textup{\protect\@secnumfont
    \ifnum\pdfstrcmp{subsection}{#1}=0 \bfseries\fi% subsection # in \bfseries
    \csname the#1\endcsname
    \protect\@secnumpunct
  }%
}  
\makeatother

\let\oldtocsection=\tocsection

\let\oldtocsubsection=\tocsubsection

\let\oldtocsubsubsection=\tocsubsubsection

\renewcommand{\tocsection}[2]{\hspace{0em}\oldtocsection{#1}{#2}}
\renewcommand{\tocsubsection}[2]{\hspace{1em}\oldtocsubsection{#1}{#2}}
\renewcommand{\tocsubsubsection}[2]{\hspace{2em}\oldtocsubsubsection{#1}{#2}}

\usepackage{hyperref,xr-hyper,mathtools}
\hypersetup{
    colorlinks=true,
    linkcolor=blue,
    citecolor=blue,
    urlcolor=blue
}

\hypersetup{
    colorlinks=true,
    linkcolor=blue,
    filecolor=magenta,      
    urlcolor=cyan,
}
\usepackage{amsmath, amsthm}
\usepackage{geometry}
\usepackage{color}                % See geometry.pdf to learn the layout options. There are lots.
\usepackage{graphicx}
\usepackage{tikz}
\usepackage{tikz-cd}
\usepackage{pgfplots}
\usetikzlibrary{matrix,arrows,decorations.pathmorphing}
\usepackage{amssymb}
\usepackage{epstopdf}
\usepackage{scalerel}
\title{H\"older estimates for degenerate complex Monge-Amp\`ere equations}

\author{ Bin Guo, S$\l{}$awomir Ko$\l{}$odziej, Jian Song and Jacob Sturm}

\usepackage{xpatch}
\makeatletter
\theoremstyle{Italic}
\begin{document}

\newcommand*{\DashedArrow}[1][]{\mathbin{\tikz [baseline=-0.25ex,-latex, dashed,#1] \draw [#1] (0pt,0.5ex) -- (1.3em,0.5ex);}}%

\maketitle

%%%%%%%%%%%%%%%%%%%%%
%% Commands %%
%%%%%%%%%%%%%%%%%%%%

\newcommand{\ssub}{\subset\joinrel\subset}

\renewcommand{\thefootnote}{\fnsymbol{footnote}}
\newcommand{\starttext}{ \setcounter{footnote}{0}
\renewcommand{\thefootnote}{\arabic{footnote}}}
\renewcommand{\theequation}{\thesection.\arabic{equation}}
\newcommand{\be}{\begin{equation}}
\newcommand{\bea}{\begin{eqnarray}}
\newcommand{\eea}{\end{eqnarray}} \newcommand{\ee}{\end{equation}}
\newcommand{\N}{{\cal N}} \newcommand{\<}{\langle}
\renewcommand{\>}{\rangle}
\def\ba{\begin{eqnarray}}
\def\ea{\end{eqnarray}}
\newcommand{\PSbox}[3]{\mbox{\rule{0in}{#3}\includegraphics{#1}
\hspace{#2}}}

\def\v{\vskip .1in}

%%%%%%%%%%%%%%%%%%%%%
%% Greek letters %%
%%%%%%%%%%%%%%%%%%%%
%\def\ll{\l}
\def\al{\alpha}
\def\b{\beta}
\def\c{\chi}
\def\d{\delta}
\def\e{\varepsilon}
\def\g{\gamma}
\def\m{\mu}
\def\n{\nu}
\def\o{\omega}
\def\f{\varphi}
\def\r{\rho}
\def\si{\sigma}
\def\t{\theta}
\def\z{\zeta}
\def\k{\kappa}

\def\G{\Gamma}
\def\D{\Delta}
\def\O{\Omega}
\def\T{\Theta}

\def\phi{\varphi}

%%%%%%%%%%%%%%%%%%%%%%%%%%%%%%%%%%%%%%%%%%%%%%%%%%%%%%%%%%%%%%%%%%%%%%%%%%%%%%%%%%%%%%%%%%%%%%%%%%%%%%%%%
%% Mathcal and Mathbb %%
%%%%%%%%%%%%%%%%%%%%%%%%%%%%%%%%%%%%%%%%%%%%%%%%%%%%%%%%%%%%%%%%%%%%%%%%%%%%%%%%%%%%%%%%%%%%%%%%%%%%%%%%%

\def\cA{{\mathcal A}}
\def\cB{{\mathcal B}}
\def\cC{{\mathcal C}}
\def\cD{{\mathcal D}}
\def\cE{{\mathcal E}}
\def\cF{{\mathcal F}}
\def\cG{{\mathcal G}}
\def\cH{{\mathcal{H}}}
\def\cI{{\mathcal I}}
\def\cJ{{\mathcal J}}
\def\cK{{\mathcal K}}
\def\cL{{\mathcal L}}
\def\cM{{\mathcal M}}
\def\cN{{\mathcal N}}
\def\cO{{\mathcal O}}
\def\cP{{\mathcal P}}
\def\cp{{\mathcal p}}
\def\cR{{\mathcal R}}
\def\cS{{\mathcal S}}
\def\cT{{\mathcal T}}
\def\cX{{\mathcal X}}
\def\cY{{\mathcal Y}}

\def\A{{\mathbb{A}}}
\def\N{\mathbb N}
\def\Z{{\mathbb Z}}
\def\Q{{\mathbb Q}}
\def\R{{\mathbb R}}
\def\C{{\mathbb C}}
\def\P{{\mathbb P}}

%%%%%%%%%%%%%%%%%%%%%%%%%%%%%%%%%%%%%%%%%%%%%%%%%%%%%%%%%%%%%%%%%%%%%%%%%%%%%%%%%%%%%%%%%%%%%%%%%%%%%%%%%
%% Roman Italic Bold %%
%%%%%%%%%%%%%%%%%%%%%%%%%%%%%%%%%%%%%%%%%%%%%%%%%%%%%%%%%%%%%%%%%%%%%%%%%%%%%%%%%%%%%%%%%%%%%%%%%%%%%%%%%

\def\Ad{{\rm Ad}}
\def\ann{{\rm ann\,}}
\def\Ass{{\rm Ass}}
\def\Aut{{\rm Aut}}
\def\depth{{\rm depth}}
\def\det{{\rm det}}
\def\dd{{\bf d}}
\def\diam{{\rm diam}}
\def\Diff{{\rm Diff}}
\def\Div{{\rm Div}}
\def\End{{\rm End}}
\def\Hilb{{\rm Hilb}}
\def\Hom{{\rm Hom}}
\def\GL{{\rm GL}}
\def\Gr{{\rm Gr}}
\def\ham{{\rm ham}}
\def\Ham{{\rm Ham}}
\def\K{{\rm K\"ahler }}
\def\Id{{\rm Id}}
\def\Ker{{\rm Ker}}
\def\KE{{\rm K\"ahler-Einstein\ }}
\def\KR{{\rm K\"ahler-Ricci }}
\def\itKR{{\it K\"ahler-Ricci }}
\def\KEE{{\rm K\"ahler-Einstein}}
\def\Lie{{\rm Lie}}
\def\mod{{\ \rm mod\ }}
\def\Mod{{\rm Mod}}
\def\ord{{\rm ord}}
\def\osc{{\rm osc\,}}
\def\PSH{{\rm PSH}}
\def\Rm{{\rm Rm}}
\def\Ric{{\rm Ric}}
\def\reg{{\rm reg}}
\def\Re{{\rm Re}}
\def\sing{{\rm sing}}
\def\SL{{\rm SL}}
\def\spec{{\rm Spec}}
\def\supp{{\rm supp}}
\def\Tr{{\rm Tr}}
\def\vol{{\rm vol}}
\def\w{{\bf w}}
\def\x{{\bf x}}
\def\y{{\bf y}}

%%%%%%%%%%%%%%%%%%%%%%%%%%%%%%%%%%%%%%%%%%%%%%%%%%%%%%%%%%%%%%%%%%%%%%%%%%%%%%%%%%%%%%%%%%%%%%%%%%%%%%%%%
%% Other macros %%
%%%%%%%%%%%%%%%%%%%%%%%%%%%%%%%%%%%%%%%%%%%%%%%%%%%%%%%%%%%%%%%%%%%%%%%%%%%%%%%%%%%%%%%%%%%%%%%%%%%%%%%%%

\def\sq{{{\sqrt{{\scalebox{0.75}[1.0]{\( - 1\)}}}}}\hskip .01in}

\def\us{{\underline s}}
\def\Im{{\rm Im}}
\def\tr{{\rm tr}}

\def\ti\tilde
\def\u{\underline}
\def\Box{{\square\,}}

\def\pl{\partial}
\def\na{\nabla}
\def\i{\infty}
\def\I{\int}
\def\p{\prod}
\def\s{\sum}

\def\ddb{\partial\bar\partial}
\def\sub{\subseteq}
\def\ra{\rightarrow}
\def\hra{\hookrightarrow}
\def\Lra{\Longrightarrow}
\def\lra{\longrightarrow}
\def\LA{\langle}
\def\RA{\rangle}

\def\half{ {1\over 2}}
\def\third{{1 \over 3}}
\def\ti{\tilde}
\def\un{\underline}

\def\pz{\partial _z}
\def\pv{\partial _v}
\def\pw{\partial _w}

\def\tet{\vartheta}
\def\dwplus{\D _+ ^\w}
\def\dxplus{\D _+ ^\x}
\def\dzplus{\D _+ ^\z}
\def\chiz{{\chi _{\bar z} ^+}}
\def\chiw{{\chi _{\bar w} ^+}}
\def\chiu{{\chi _{\bar u} ^+}}
\def\chiv{{\chi _{\bar v} ^+}}
\def\os{\omega ^*}
\def\ps{{p_*}}

\def\hO{\hat\Omega}
\def\ho{\hat\omega}

\def\fr{\frak}
\def\[{{\bf [}}
\def\]{{\bf ]}}
\def\Rd{{\bf R}^d}
\def\Ci{{\bf C}^{\infty}}
\def\pl{\partial}
\def\sq{{{\sqrt{{\scalebox{0.75}[1.0]{\( - 1\)}}}}}\hskip .01in}
\newcommand{\dotcup}{\ensuremath{\mathaccent\cdot\cup}}

%%%%%%%%%%%%%%%%%%%%%%%%%%%%%%%%%%%%%%%%%%%%%%%%%%%%%%%%%%%%%%%%%%%%%%%%%%%%%%%%%%%%%%%%%%%%%%%%%%%%%%%%%
%% New commands %%
%%%%%%%%%%%%%%%%%%%%%%%%%%%%%%%%%%%%%%%%%%%%%%%%%%%%%%%%%%%%%%%%%%%%%%%%%%%%%%%%%%%%%%%%%%%%%%%%%%%%%%%%%

\newenvironment{thm}{\begin{theorem}}{\end{theorem}}
\newcommand{\ssubset}{\subset\joinrel\subset}
\newtheorem{theorem}{Theorem}
\newtheorem{corollary}{Corollary}
\newtheorem{lemma}{Lemma}
\newtheorem{definition}{Definition}
\newtheorem{proposition}{Proposition}
\newtheorem{remark}{Remark}
\newtheorem{pr}{Problem}
\begin{abstract}  
{\footnotesize }
Uniform $L^\infty$ and H\"older estimates were proved by the second named author for complex Monge-Amp\`ere equations on compact K\"ahler manifolds with $L^p$ volume measure with $p>1$. On the other hand, establishing H\"older estimates on singular K\"ahler varieties has remained open.  In this paper, we establish uniform H\"older continuity for a family of complex Monge-Amp\`ere equations on K\"ahler varieties, by developing a geometric regularization based on the partial $C^0$ estimate, i.e., quantitive Kodaira embeddings. As an application, we prove that local potentials of smoothable K\"ahler-Einstein varieties are H\"older continuous. This is the first installment of our project on the quantitive relation between intrinsic and extrinsic metric structures on singular K\"ahler varieties.

\end{abstract}

\parindent=0in
\setcounter{equation}{0}

\section{Introduction}
\v
Let $X$ be a compact complex manifold of dimension $n$, and $\o$ a \K metric on $X$. We say that $\o$ is a K\"ahler-Einstein metric
if it satisfies the equation $\Ric(\o)=\lambda\o$ for some $\lambda\in\{-1,0,1\}$. Uniqueness of such metrics was proved by Calabi \cite{Cal} in the cases $\lambda=-1,0$ and Bando-Mabuchi \cite{BM} in the case $\lambda=1$.  Existence was settled much later: The case $\lambda=-1$ was proved by Aubin \cite{A} and Yau \cite{Y}. The case $\lambda=0$ was achieved in Yau's celebrated proof of the Calabi-Yau conjecture \cite{Y}. The case $\lambda=1$  was proved in the deep and influential papers of Chen-Donaldson-Sun
\cite{CDS1, CDS2, CDS3} and Tian \cite{T3}.

\v
Central to these existence proofs is the establishment of a priori estimates. When $\lambda=0$, a critical first step is the $C^0$ a priori estimate, which can be formulated as follows. Consider the complex Monge-Amp\`ere equation

\v
\be\label{MA}   \o^n\ = \ (\o_0+\sq\ddb\phi)^n\ = \ e^F\o_0^n,\ \ \ \ \ \o_0+\sq\ddb\phi >0,\ \ \ \ \ \sup\phi=0.
\ee
\v

\v
Here $\o_0$ is a smooth background metric on  $X$ and $F$ is a smooth function satisfying $\I_X e^F\o_0^n=\I_X\o_0^n$.   Yau \cite{Y} showed that (\ref{MA}) admits a smooth  solution $\phi$. The first step in the proof is an a priori estimate which states that $\|\phi\|_{C^0(X)}\leq C(n,\o_0,\|F\|_{L^\i(X)})$. 
\v
A major advance in extending these estimates was achieved by 
Ko\l odziej (\cite{K,K1}), who, using techniques of pluripotential theory, was able to establish both $L^\infty$ and $C^\al$ a priori estimates for the solutions under weaker assumptions on $F$. More precisely, for any $p>1$ there exists $\al>0$ such that $\|\phi\|_{C^\al(X)}\leq C(n,\o_0,\|e^F\|_{L^p(X)})$. The $L^\infty$ estimate is extended to the degenerate case for equation (\ref{MA}) when $[\omega_0]$ is a big and semi-positive class by \cite{EGZ}. Naturally one would also like to extend the H\"older estimate to the degnerate setting.

 \begin{pr} Suppose the smooth closed $(1,1)$-form $\omega_0$ is  nonnegative with 
 $$\int_X \omega_0^n \geq 0 . $$  The equation (\ref{MA}) still admits a unique solution $\phi \in L^\infty(X)\cap PSH(X, \omega_0)$.  Is $\phi$ still H\"older continuous, i.e., $\phi \in C^\alpha(X)$ for some $\alpha>0$? 
 \end{pr}

\v
Equation (\ref{MA}) is degenerate when $[\omega]$ is not K\"ahler. Such degenerate equations are essential to the study of canonical metrics such as K\"ahler-Einstein metrics  on a singular K\"ahler variety $X$  because the complex Monge-Amp\`ere equations become degenerate after being pulled back on the non-singular model of $X$.  In recent years, there has been a growing interest in  K\"ahler-Einstein metrics on singular varieties, motivated in part by their role  in the compactification of moduli space and their connection to K-stability (c.f.  \cite{T1, DS, CDS1, CDS2, CDS3, T3, Sz2}).   The existence theory for K\"ahler-Einstein currents on singular varieties is fairly complete, However, unlike the case of smooth manifolds, a priori estimates for singular K\"ahler-Einstein metrics and their behavior near the singular set are largely unknown.
The stronger regularity, particularly the H\"older continuity for the K\"ahler-Einstein potentials with respect to smooth K\"ahler metrics on $X$,  reveals the deep connections between intrinsic (the K\"ahler-Einstein metrics) and extrinsic (the Bergman metrics via projective embeddings) geometric structures on $X$. For example, the H\"older estimate for complex Monge-Amp\`ere equation  reveals the H\"older comparison between the background metric and the K\"ahler metric induced from the equation \cite{LY1}.
\v

The goal of this paper is to establish various H\"older estimates for  complex Monge-Amp\`ere equations on singular K\"ahler varieties under suitable geometric assumptions. We will first develop a uniform H\"older estimate for polarized manifolds with Ricci curvature lower bound and diameter upper bound. Fix $n,d\geq 1$. We define  $\mathcal{F}(n)$ to be the set of all quadruples 
$(X,\o_X, L, h_L)$ satisfying
\v
\begin{enumerate} 
\item $(X,\o_X)$ is a compact \K manifold of dimension $n$,  
\vskip .04in
\item $\Ric(\o_X)\geq {-1}$,  
\vskip .04in
\item $L\ra X$  a  holomorphic line bundle with metric $h_L$ satisfying
$-\sq\ddb\log h_L=\o_X$. 

\end{enumerate}

and $\cF(n,d)=\{(X,\o_X, L, h_L)\in \cF(n)\,:\, 
\diam(X,\o_X)\leq d)$. We also define $\cF^0(n)=
\{(X,\o_X, L, h_L)\in \cF(n)\,:\, \Ric(\o_X)\leq 1\}$
and similarly for $\cF^0(n,d)$

%We shall sometimes write $(X,L)$ in place of $(X,\o_X,L,h_L)$ to lighten the notation.

\v
The family $\mathcal{F}(n,D)$ is bounded: According to a theorem of Donaldson-Sun \cite{DS} and Liu-Szekelyhidi \cite{LZ}, there exists
$k(n,D)\geq 1$ such that the map $T_k: X\ra \P^{N_k}$, given by an orthonormal basis of $H^0(X, kL)$, is
an embedding. If $\T_{\rm FS}$ is the Fubini-Study metric on $\P^{N_k}$, we define 
\begin{equation}\label{fsm}
\o_{\rm FS}= {1\over k}T_k^*\T_{\rm FS}
\end{equation} and we let $\phi_X$ be defined by the equation 
\begin{equation}\label{locpo}
\o_X=\o_{\rm FS}+\sq\ddb\phi_X, ~\sup\phi_X=0. 
\end{equation}

The following is our first result.

\v

\begin{theorem}\label{m}
For any $n$, $D>1$, there exist $k(n, D)>0$, $C(n,D)>0$ and $\al(n,D)>0$  such that for any $(X, \o_X, L, h_L) \in \mathcal{F}(n, D)$ and $x, y\in X$, 
\vskip .001in 
\be\label{holder}|\phi_X(x)-\phi_X(y)|\leq Cd_{\rm FS}(x,y)^{\al}, 
\ee
\v
where $\phi_X$ is defined in \ref{locpo} and $d_{\rm FS}$ is the extrinsic distance function on $X$ associated to $\o_{\rm FS}$.
\end{theorem}

The idea of the proof is the global approximation of  $\phi_X$ by  Bergman potentials as in \cite{Don, Song1, Song2, PS1, PS2} combined with the partial $C^0$-estimate, i.e., quantitive Kodaira embeddings developed in \cite{T1, DS, LZ} and a finite generation scheme in  \cite{L}.

\v
Remark: If $r$ is a positive integer, we have the multiplication map 
$\cF(n,d)\ra \cF(n, rd)$ defined by $r\cdot(X,\o_X,L,h_L) =(X,r\o_X,rL,h_L^r)$. Let $\cF_r(n,d)$ be the image of this map.  If $r=r(n,d)\geq 1$ we may, in the proof of (\ref{holder}), replace $\cF(n,d)$ by  $\cF_r(n,d)$. We will make this type of replacement repeatedly throughout the proof.

\v
We will apply Theorem \ref{m} to  K\"ahler currents on smoothable singular projective varieties. Let us  recall some definitions before stating our results.

\v

\begin{definition}
Let $X$ be an $n$-dimensional projective variety with klt singularities.  $X$ is said to be smoothable if there is a flat family 
\begin{equation}\label{flatf}
\pi:\cX \subset \mathbb{P}^N \ra\D
\end{equation}
 over the unit disk $\D\sub\C$ such that 
\v
\begin{enumerate}

\item $X=\cX_0$ is the central fiber  and $\pi:\cX\backslash X_0\ra \D\backslash\{0\}$ is smooth, 

\medskip

\item 
$\mathcal{X}$ is $\mathbb{Q}$-Gorenstein and $K_{\cX/D}$ is $\Q$-Cartier. 

\end{enumerate}

\end{definition}

\v
\begin{definition} Let $X$ be a normal projective variety with klt singularities. A closed positive $(1,1)$-current   $\o_{\rm KE}$ with bounded local potentials is said to be a K\"ahler-Einstein current if it satisfies the following. 
\v
\begin{enumerate}

\item $[\o_{\rm KE}] \in H^2(X, \mathbb{Q})$ is an ample $\mathbb{Q}$-Cartier divisor. 

\medskip

\item Let $h_{KE} = (\o_{\rm KE})^n$ be the hermitian metric on $-K_X$. Then $Ric(h_{KE})= -\ddb \log h_{KE}$ is a well-defined closed $(1,1)$-current with 
$$Ric(h_{KE}) = \lambda \o_{\rm KE}$$
 for some $\lambda\in \{-1,0, 1\}$. 
 
\end{enumerate}

We say that $(X, \omega_{\rm KE})$ is a projective K\"ahler-Einstein variety. 

\end{definition}

For $\lambda = -1, 0, 1$, $h_{KE}$ corresponds to a negative, flat and positive hermitian metrics. By the smoothing property of the weak K\"ahler-Ricci flow \cite{ST3}, it is shown in \cite{SzT}, $\omega_{\rm KE}$ is in fact smooth on $X^{\rm reg}$, the regular part of $X$.   

\v

%We say that $(X, \omega_{KE})$ is KE-smoothable if there exist a sequence of $n$-dimensional polarized \KE manifolds $(X_i, g_i)$, $i=1, 2, ...$,  satisfying the following 

%\begin{enumerate} 

%\item $Ric(g_i) = \lambda g_i$, where $\lambda =-1$, $0$ or $1$, 

%\medskip

%\item $(X_i, g_i)$ converge globally in Gromov-Hausdorff distance, 

%\medskip

%\item $(X_i, g_i)$ converge smoothly to $(X^{\rm reg}, \omega_{KE})$. 

%\end{enumerate}

%In fact, the metric completion of $(X^{\rm reg}, \omega_{KE})$ is isomorphic to the Gromov-Hausdorff limit of $(X_i, g_i)$ and is homeomorphic to the projective variety $X$ itself. The work of \cite{DS} implies that any $n$-dimensional KE-smooth variety $(X, \omega_{KE})$ arises as the Gromov-Hausdorff limit for a sequence of K\"ahler-Einstein manifolds in $\mathcal{F}(n, D)$ for some $D>0$. 

Our next result gives the H\"older continuity for K\"ahler-Einstein potentials on smoothable K\"ahler-Einstein variety. 

\v
\begin{theorem}\label{H} Let  $(X,\o_{\rm KE})$ be a smoothable projective K\"ahler-Einstein variety as the central fibre the smoothing family $\pi: \mathcal{X}\subset \mathbb{P}^N \rightarrow \D$.  Let $\theta\in [\omega_{\rm KE}]$ be a smooth background metric on $X$ and 
 $$\o_{\rm KE} = \theta + \ddb \phi_{\rm KE}. $$ 
  Then  there exist $C,\gamma>0$ such that
\be\label{holderke} 
|\phi_{{\rm KE}}(x)-\phi_{{\rm KE}}(y)|\leq C d_{\mathbb{P}^N }(x,y)^{\gamma}
\ee
 where $d_{\mathbb{P}^N}$ is the distance between $x$ and $y$ in $\mathbb{P}^N$.
\end{theorem}

Theorem \ref{H} confirms a conjecture in \cite{GGZd} in the case of smoothable K\"ahler-Einstein varieties.
\v
Remark: One can not expect \KE metrics on singular varieties to be more regular than H\"older continuous, as can already be seen in dimension two.

\v

We will now return to Problem 1 and study H\"older continuity for complex Monge-Amp\`ere equation on smoothable projective varieties. Suppose $X$ is an $n$-dimensional smoothable projective variety with klt singularities, where $X=\mathcal{X}_0$ in the flat family (\ref{flatf}). Since $X$ has klt singularities, we can choose an adaptive volume measure $\Omega$ on $X$, i.e., locally there exist $m \in \mathbb{Z}^+$,  a generator $\eta$ of  $mK_X$ and a smooth function $f$ such that 
$$\Omega = e^f |\eta|^{2/m}.$$
Let $\Theta$ be the Fubini-Study metric on $\mathbb{P}^N$ and $\theta_t = \Theta|_{\mathcal{X}_t}$ for $t\in \D$.  We can consider the following complex Monge-Amp\`ere equation on $X$
\begin{equation}\label{mafl}
(\theta_0 + \sq\ddb \phi)^n = e^F \Omega, ~ \sup_X \phi=0.
\end{equation}
where $\phi \in PSH(X, \theta_0)$ and $\int_X e^F \Omega = [\theta_0]^n$. It is proved in \cite{EGZ} that equation (\ref{mafl}) admits a unique continuous solution if $e^F \in L^p(X, \Omega)$ for some $p>1$. Naturally, one would hope that $\phi$ is H\"older continuous with respect to $\theta_0$ or $d_{\mathbb{P}^N}$ (the distance in $\mathbb{P}^N$). 

\v

\begin{theorem}\label{mainthm3} Let $X$ be an $n$-dimensional smoothable projective variety with klt singularities as the central fibre the smoothing family $\pi: \mathcal{X}\subset \mathbb{P}^N \rightarrow \D$. Suppose there exists $p>1$ and $A>0$ such that 
\begin{enumerate}

\item $e^F \in L^p(X, \Omega)$, 

\medskip

\item $-F\in PSH(X, A \theta_0)$, i.e., $A\theta_0 - \sq\ddb F \geq 0$. 

\end{enumerate}
Then there exists a unique H\"older continuous (with respect to $d_{\mathbb{P}^N}$) solution $\phi$ to the equation (\ref{mafl}). More precisely, there exist $C>0$ and $\gamma >0$ such that for any $x, y \in X$, 
\begin{equation}\label{hodsmo}
|\phi(x) - \phi(y)| \leq C d_{\mathbb{P}^N}(x, y)^\gamma, 
\end{equation}
where $C$ and $\gamma$ depend on $X$, $\theta_0$, $\Omega$, $p$, $A$ and $||e^F||_{L^p(X, \Omega)}$.

\end{theorem}

\v

Let $\omega= \theta_0 + \ddb \varphi$ for the solution $\varphi$ to equation (\ref{mafl}). In addition to the algebraic smoothing condition for the underlying variety $X$, Theorem \ref{mainthm3} imposes a positivity assumption for the Monge-Amp\`ere measure $\omega^n$, i.e., 
\begin{equation}\label{riclowfa}
Ric(\omega)=-\ddb \log \omega^n = Ric(\Omega) - \ddb F \geq - C\theta_0
\end{equation}
for some $C>0$ since $Ric(\Omega)=-\ddb \log \Omega $ is bounded below by some negative propotion of $\theta_0$. The positivity assumption (\ref{riclowfa}) is equivalent to a uniform lower bound for the Ricci curvature of $\omega$.

\begin{corollary} Let $X$ be an $n$-dimensional smoothable projective variety with klt singularities with klt singularities as the central fibre the smoothing family $\pi: \mathcal{X}\subset \mathbb{P}^N \rightarrow \D$. Suppose $F\in C^\infty(X)$. 
Then there exists a unique H\"older continuous (with respect to $d_{\mathbb{P}^N}$) solution $\phi$ to the equation (\ref{mafl}).

\end{corollary}

\v 

There are various extensions of Theorem \ref{mainthm3} that will be investigated in forthcoming works. For example, the smoothable assumption in Theorem \ref{mainthm3} can be removed in complex dimension $3$ as well as in high dimensions for normal varieties that admit special resolution of singularities. 

%%%%%%%%

%Our convention for the inner product $\Hilb(h_L^m) = \Hilb(m)$ is as follows:
\v
%$$ \<\si_\al,\si_\b\>_{\Hilb(h_L^m)}\ = \ \I_X \si_\al\bar \si_\b h_L^m\cdot{(m\o_X)^n\over n!} $$
\v
\v

%%%%%%%%%%%%%%%%%%%%%%%%%%%%%%%%%%%%%%%%%%%%%%%%%%

\v
\section{Partial $C^0$ estimate and effective finite generation}
\setcounter{equation}{0}
 \v
 We shall make use of the partial $C^0$ estimate of Zhang \cite{Z},
 which is extends earlier work of  \cite{T2, DS, LZ}.
 
  \v

 We will generally use $s$ to denote a section of $L$ and $\si$ to denote of section of $mL$ for $m>1$.
 \v
 After replacing $\cF(n)$ by $\cF_r(n)$ for an appropriate $r\geq 1$, we may assume $L\ra X$ is very ample for all $(X,\o_X,L,h_L)\in \cF(n)$. Let $(X,\o_X,L,h_L)\in \mathcal{F}(n)$ and define, for each $m\geq 1$  the Bergman kernel

$$\r(X,m\o,mL,h_L^m)(x):= \ \s_{j=0}^N |\si_j|^2_{h_L^m}
$$
where $\{\si_0,,,,.\si_{N_m}\}$ is an orthonormal basis of
$H^0(X,mL)$, i.e.

$$\<\si_i,\si_j\>_{\rm Hilb(h_L^m)}\ :=\ \I_X\, \si_i\bar\si_j\,h_L^m\,(m\o)^n\ = \ \d_{ij}
$$

 \v
 \begin{thm}{\rm (\cite{Z}).} \label{KZ} Given $n\geq 1$ and $S>0$ there exist $r=r(n,S)\in\N$ and $ A(n,S)>1$ with the following property.
 Let $(X,\o_X,L,h_L)\in \mathcal{F}_r(n)$ and assume $C(\o_X)\leq S$  where $C(\o_X)$ is the Sobolev constant of $\o_X$. Then for any $m\in\N$ we have 
 \be\label{pco}
 A^{-1}\leq \r(X,m\o_X,mL,h_L^m)\leq A
 \ee

   \end{thm}
The condition  $C(\o_X)\leq S$ means
\v
$$ \left(\I_X|u|^{2n\over n-1}\o_X^n\right)^{n-1\over n}
\ \leq \
S\left(\I_X |u|^2\o_X\ + \ \I_X |\na u|^2\,\o_X^n 
\right),\ \ \forall{u}\ \in\ W^{1,2}(X)
$$
 \v
\begin{corollary}\label{pc} 
Let $n,d>0$. Then there exists $A=A(n,d)>1$ and $r=r(n,d)>1$ satisfying (\ref{pco}) for all $(X,\o_X,L,h_L)\in \cF_r(n,d)$.
\end{corollary}

This follows from Croke's theorem \cite{C} which guarantees that the Sobolev constants of $m\o_X$ are uniformly bounded for $(X,\o_X,L,h_L) \in \cF(n,d)$ and $m\geq 1$. 

\v
Next we recall a theorem of Li \cite{L}.

\v
\begin{theorem}{\rm \cite{L}}. \label{L} Let $k\geq 1, l\geq 1$ and $m\geq (n+2+l)k$. 
Let
$(X,\o_X,L,h_L)\in \mathcal{F}(n,d)$,  and
$s_0\cdots s_{N_k}$  an orthonormal basis of $H^0(X,kL)$ with respect to ${\rm Hilb(h_L^k)}$. In particular, $N_k+1=\dim H^0(X,kL)$.  Then for $U\in H^0(X,mL)$, we can write
\v
\be\label{1}
U\ = \ \s_{\al_1,...,\al_l=0}^{N_k}\,U({\al_1,...,\al_l})\,s_{\al_1}\cdots s_{\al_l}
\ee
\v
where \v
$$ U({\al_1,...,\al_l})\in H^0(X,(m-kl)L),  \ 
$$
\v
$$ \|U({\al_1,...,\al_l})\|^2_{\rm Hilb(h_L^{(m-lk)})}\ \leq \ 
{
(n+l)!
\over
n!l!
}
{(m-kl)^n\over m^n
}
A^{2n+2+l}\,\|U\|^2_{{\rm Hilb(h_L^m)}}
$$
\end{theorem}

\v

\begin{corollary} \label{h1}
Let $p>0$ be a multiple of $(n+2)$ and let $(X,L)\in\cF_p(n)$. Then

$$ \bigoplus_{m=0}^\i H^0(X, mL)
$$
is generated, as a $\C$-algebra, by $H^0(X,L)$
\end{corollary}

\v
Remarks.  Li's normalization of Hilb is a bit different, so in his formulation, the factor
${(m-kl)^n\over m^n
}
$ is missing. Also, in his proof of Theorem \ref{L}, he assumes $\Ric(\o)>0$ and proves a slightly stronger estimate. The proof in  the case $\Ric(\o)\geq -\o/2$ requires very minor modifications. For the sake of completeness, we provide it below.
\v
First we recall Siu's global version of the Skoda division theorem.
\v
\begin{theorem}{\rm (\cite{Siu})}\label{Siu}
Let $X$ be a compact projective manifold of dimension $n$, let $G\ra X$ and $E\ra X$ be holomorphic line bundles and 
$e^{-\psi_E}$ a hermitian metric on $E$ with $\ddb\psi_E>0$. Let $l\geq 1$ and $s_0,...,s_N\in H^0(X,G)$. Let 
$|s|^2=\s_{\al=0}^N|s_\al|^2$. If
\v 
$$\hbox
{$U\in H^0(X, (n+l+1)G +E+K_X)$\ \  satisfies} \ \ 
 \I_X{|U|^2e^{-\psi_E}\over |s|^{2(n+l+1)}} \ < \ \i
$$
\v
then $U=\s_{\al=0}^N U_\al s_\al$ with
\vskip .03in
$$U_0,...,U_N\in H^0(X,((n+l)G+E+K_X)\ \ {\rm and} \ \ 
\I_X{|U_\al|^2e^{-\psi_E}\over |s|^{2(n+l)}}\leq {n+l\over l}
\I_X{|U|^2e^{-\psi_E}\over |s|^{2(n+l+1)}}
$$
\end{theorem}
\v
Here we define $|s|^{2(n+l+1)}:= \left(\s_{\al=0}^N |s_\al|^2\right)^{(n+l+1)}$.
\v

{\it Proof of Theorem 2.} Let
$(X,\o_X,L,h_L)\in \mathcal{F}(n,d)$ and write (locally)\ $h_L=e^{-\phi_L}$ and $\o_X=\sq\ddb\phi_L$. Then we can write

$$\bf \textcolor{magenta}{mL}\ = \ \textcolor{red}{(n+l+1)(kL)}\ + \ \textcolor{blue}{((m-(n+l+1)k)L - K_X)}\ + \ \textcolor{brown}{K_X}:=
\textcolor{red}{(n+l+1)G}+\textcolor{blue}{E}+\textcolor{brown}{K_X}
$$
\v
$$\psi_E\ = \ (m - (n+l+1)k)\phi_L\ - \ \log{\o_X^n\over n!}\hskip.7in
$$
Then
$$ \sq\ddb\psi_E\ = \ (m - (n+l+1)k)\o_X \ + \ \Ric(\o_X) \geq\  \ \o_X\ + \ \Ric(\o_X)\ \geq \ \o_X/2
$$
\v
since $(m - (n+l+1)k)\geq 1$ by assumption.
Note that if $0\leq r\leq l$ we have

\v
\be\label{rew} e^{-\psi} \ = \ {e^{-(m-rk)\phi}\over  
e^{-([n+l+1]-r)k\phi} } \cdot{\o^n\over n!}\ \ {\rm and}\ \ 
e^{-([n+l+1]-r)k\phi} |s|^{2([n+l+1]-r)}\ = \ \r_{k}^{([n+l+1]-r)}
\ee

\v
where $\r_k=\r(X,k\o_X,h_L^k,kL)$.
Let $U\in H^0(X,mL)$ and let $s_0,...,s_{N_k}$ be an orthonormal basis of $kL$. Then Theorem \ref{Siu} immplies that 
\v
$$ U\ = \ \s_{\al_1=0}^{N_k} U(\al_1;1)s_{\al_1}\ 
{\rm and} \ \ 
\I_X{
|U(\al_1;1)|^2e^{-\psi}\over |s|^{2[(n+l+1)-1]}} \leq {[n+l+1]-1\over [l+1]-1}\I_X{|U|^2e^{-\psi}\over |s|^{2(n+l+1)}}
$$
\v
where $U(\al_1;1)\in H^0(X, (m-k)L)$ for $0\leq \al_1\leq N_k$. Applying Theorem \ref{Siu} again:

$$ U(\al_1;1)\ = \ \s_{\al=0}^{N_k} U(\al_1,\al_2;2)s_{\al_2}\ 
{\rm and} \ \ 
\I_X{
|U(\al_1,\al_2;2)|^2e^{-\psi}\over |s|^{2([n+l+1]-2)}} \leq {[n+l+1]-2\over [l+1]-2}\I_X{|U(\al;1)|^2e^{-\psi}\over |s|^{2([n+l+1]-1)}}
$$
with $U(\al_1,\al_2;2)\in H^0(X, (m-2k)L)$.\
Continuing in this fashion, we obtain

$$ U(\al_1,...,{\al_{r-1}};1)\ = \ \s_{\al_r=0}^{N_k} U(\al_1,\al_2,...,\al_r;r)s_{\al_r}\ 
{\rm and} \ \ 
$$
\v
$$
\I_X{
|U(\al_1,\al_2,...,\al_r;r)|^2e^{-\psi}\over |s|^{2([n+l+1]-r)}} \leq {[n+l+1]-2\over [l+1]-r}\I_X{|U(\al_1,...,\al_{r-1};r-1)|^2e^{-\psi}\over |s|^{2([n+l+1]-r)}}
$$
\v
Applying this step for $r=1,2,...,l$ and using (\ref{rew}) we obtain
\v
$$
\I_X{
|U(\al_1,\al_2,...,\al_l;l)|^2e^{-(m-1k)\phi}\over \r_{k}^{n+1}} {\o^n\over n!}
\leq {(n+l)!\over n!l!}\I_X{|U|^2e^{-m\phi}\over 
\r_{k}^{(n+l+1)}}\cdot{\o^n\over n!}
$$
\v
Finally, estimate  $\r_{k}^{n+1}$ on the left by $A^{n+1}$ and $\r_{k}^{n+l+1}$ on the right by $A^{-(n+l+1)}$ \qed
\v
{\it Proof of Corollary \ref{h1}}.
If $(X,L')\in\cF_p(n,d)$ there exists $(X,L)\in\cF(n,d)$ such that  $L'~=~pL$.
Then
our goal becomes the following: Let $r\geq 1$ and   
$U\in H^0(X, (r+1)L')= H^0((r+1)pL)$. We must show that $U$ is a linear combination of terms of the form $U_0U_1\cdots U_{r}$ where $U_j\in H^0(X,L')=H^0(X,pL)$.

\v

 Choose $k$ such that $p=(n+2)k$ and let $l=(n+2)r$. Then
 $m :=  (r+1)p = (n+2+l)k$. In particular, $U\in H^0(X, mL)$ and $m-lk=p$.
 
 \v
 
  Now by Theorem \ref{L} we can write

$$ U\ = \ \s_{\al_1,...,\al_l=1}^{N_k} U(\al_1,...,\al_l)s_{\al_1}\cdots s_{\al_l}
$$

Since $l=(n+2)r$ we can write 
$s_{\al_1}\cdots s_{\al_l}\ = \ U_1\cdots U_r$ with $U_j\in H^0(X, pL)$. \qed

\v
\section{Weak effective finite generation}
\setcounter{equation}{0}
\v

In Theorem \ref{L} the coefficients $U({\al_1,...,\al_l})$ in (\ref{1}) are sections and the bounds they satisfy are uniform and effective. We  need another version of this result in which we allow the coefficients to be real numbers, and with bounds that are uniform but not effective.

\v

\begin{proposition}\label{weak} Let $n,d,m\geq 1$. Then there exist $\lambda_m>0$ and $r=r(n,d)$  with the following properties.
Let
$(X,\o_X,L,h_L)\in \mathcal{F}_r(n,d)$, and
$s_0\cdots s_{N}$  an orthonormal basis of $H^0(X,L)$ with respect to ${\rm Hilb}(h_L)$. Let $U\in H^0(X,mL)$. Then we can write
\v
\be\label{2}
U\ = \ \s_{\al_1,...,\al_m=0}^{N}\,U({\al_1,...,\al_m})\,s_{\al_1}\cdots s_{\al_m}
\ee
\v
where \v
$$ U({\al_1,...,\al_m})\in \R,  \ 
$$
\v
\be\label{lamdam} |U({\al_1,...,\al_m})| \ \leq \ \lambda_m
\|U\|_{{\rm Hilb(h_L^m)}}
\ee
\end{proposition}
The proof will be via contradiction and makes use some results of \cite{DS} which we now recall.
\v

\subsection{Gromov-Hausdorff Limits of K\"ahler Manifolds} We state here some results from \cite{DS} that will be needed in the proof of Proposition \ref{weak}. There exists $r=r(n,d)$ with the following properties. Let $(X_i,\o_i,L_i,h_i)\in \cF_r^0(n,d)$. After passing to a subsequence, the sequence converges, in the Gromov-Hausdorff sense, to a limit $(X_\i,\o_\i,L_\i,h_\i)$. Here $X_\i$ is a metric space whose regular part $X_\i^{\rm reg}\sub X_\i$ is an open dense complex manifold of dimension $n$, $\o_\i$ is a \K metric on  $X_\i^{\rm reg}$ and $L_\i$ is a holomorphic line bundle on $X^{\rm reg}$ with metric $h_\i$  and  curvature $\o_\i$. The meaning of the convergence 
$(X_i,\o_i,L_i,h_i)\ra (X_\i,\o_\i,L_\i,h_\i)$ is as follows:
\v

There are open sets $U_i\sub U_{i+1}\sub\cdots\sub X^{\rm reg}_\i$, exhausting $X^{\rm reg}$ with relatively compact containments,  open sets $V_i\sub X_i$, and diffeomorphisms $\G_i:U_i\ra V_i$ which are approximate \K isometries, in the sense that $\o_i\ra\o_\i$ or, more precisely, $\|\G_i^*\o_i\ra \o_\i\|_{C^\i(\O)}\ra 0$  for any relatively compact open set $\O\sub U_i$. Similarly $J_i\ra J_\i$ and $g_i\ra g_\i$. We will usually omit $\G_i$ in what follows to lighten the notation and similarly we will identify $U_i$ with $V_i$. Furthermore, $\O\sub U_i$ will always denote a relatively compact open set.
\v
We also have $L_i\ra L_\i$. This means there is a locally finite open cover $X^{\rm reg}_\i=\cup_\al A_\al$
and holomorphic sections $s^\i_\al\in H^0(A_\al, L_\i)$ and $s^i_\al\in H^0((A_\al\cap U_i), L_i)$ such that the sequence of smooth cocycles $[(s^i_\al)(s^i_\b)^{-1}]\circ\G_i$  converges uniformly to the holomorphic cocycle
$[(s^\i_\al)(s^\i_\b)^{-1}]$  on $A_\al\cap A_\b\cap\O$.
Similarly, if $\si_\i\in H^0(X_\i,mL_\i)$ and $\si_i\in H^0(X_i,mL_i)$ we say
$\si_i\ra \si_\i$ if $\si_i/(s^i_\al)^m\ra \si_\i/(s^\i_\al)^m$ uniformly on $A_\al\cap \O$. If $\si_i\ra \si_\i\in H^0(X_\i,mL_\i)$
and $\tau_i\ra \tau_\i\in H^0(X_\i,lL_\i)$ one easily sees

\be\label{products}  \si_i\tau_i\ \ra \ \si_\i\tau_\i\in H^0(X_\i,(m+l)L_\i)
\ee
\v
There is an isometry $Q_i^m: H^0(X_\i,mL_\i)\ra H^0(X_i,mL_i)$ such that $Q_i^m(s)\ra s$ as $i\ra\i$. In particular, if $\{s_{0}^\i,...,s_{N}^\i\}$ is an orthonormal basis of $H^0(X_\i,L_\i)$, then
$\{s_{0}^i,...,s_{N}^i\}$ is an orthonormal basis of $H^0(X_i,L_i)$,
where $s_{\al}^i:= Q_i(s_{\al}^\i)$.

\v
\v
For $1\leq i\leq \i$ we
let $T_i: X_i\hookrightarrow\P^N$ be the embedding of $H^0(X_i,L_i)$ with respect to a orthonormal basis (so $T_i$ is only defined up to the action of $U(N+1)$). Then there is a compact metric space $(X_\i, d_\i)$ such that, after passing to a subsequence if necessary, $(X_i,\o_i)\ra (X_\i, d_\i)$ in the Gromov-Hausdorff topology. In addition, $T_i(X_i):= W_i$ is a normal sub-variety of $\P^N$ (smooth if $i<\i$), such that $W_i\ra W_\i$ algebraically, i.e. the corresponding point converge in the Hilbert scheme. Moreover, we have $T_i(x_i)\ra T_\i(x_\i)$ in $\P^N$ if $x_i\ra x_\i$. The map $T_\i: X_\i\ra W_\i$ is a homeomorphism - in fact it is a biholomorphic map with respect to the natural structure of complex variety on $X_\i$.

 \v
 
 We summarize this discussion with the following diagram:

\be\label{diag}
\begin{tikzcd}
  & X_i\ \arrow{r}{T_i}\arrow{d}[swap]{\rm CC} 
  &\  W_i\ \arrow[hookrightarrow]{r}{}\arrow{d}{\rm Hilb} 
  & \ \P^N 
  & {} \\
  & X_\i\ \arrow{r}{T_\i} 
  &\ W_\i\ \arrow[hookrightarrow]{r}[swap]{} 
  & \ \P^N  
  \end{tikzcd}
\ee
 
Here the vertical arrows represent convergence in the metric (Cheeger-Colding) sense and the the algebraic (Hilbert scheme) sense respectively. The  horizontal arrows isomorphisms:  $T_i$ is an algebraic isomorphism, and $T_\i$ is a holomorphic isomorphism.  For $1\leq i\leq \i$, the maps $W_i\hookrightarrow\P^N$ are inclusions.
\v\v
\subsection{Proof of Proposition \ref{weak}}  We will give the argument for $\cF_r^0(n,d)$. We don't need to use the fact that the singular sets of $X_\i$ and $W_\i$ match, so proof of the proposition for 
$\cF_r(n,d)$ is the same, using instead the results of \cite{LZ}. Assume not. Let $\{s_0^\i,...,s_N^\i\}$ be an orthonormal basis for $H^0(X_\i,L_\i)$ and $s_p^i=Q_i(s_p^\i)$. Then there exists $m\geq 1$ and $(X_j,\o_j,L_j,h_j)\in \mathcal{F}(n,d)$ and $U_j\in H^0(X,mL)$ with $\|U_j\|_{\rm Hilb(h^m)}=1$ satisfying the following. If

$$
U_j\ = \ \s_{\al_1,...,\al_m=0}^{N}\,U_j({\al_1,...,\al_m})\,s^{(j)}_{\al_1}\cdots s^{(j)}_{\al_m}\ = \ \s_\al U_j^\al s^{(j)}_\al
$$

for some $U_j({\al_1,...,\al_m})\in\R$, then $\max_\al \{|U_j(\al)|\} \geq j$.

\v

\v
After passing to a subsequence, the theorem of \cite{DS} says that $X_i\ra X_\i$ in the Gromov-Hausdorff topology, and if for $1\leq i\leq \i$,  $T_i: X_i\hookrightarrow\P^{N}$ is an embedding by the orthonormal basis
$\{s_{0}^\i,...,s_{N}^i\}
$
 of $H^0(X_i, L_i)$, and $W_i\ra W_\i$ in the algebraic sense. \v

After passing to a further subsequence, there exists
$U_\i\in H^0(X,mL)$ such that $U_j\ra U_\i$. This holds since after scaling and a unitary transformation we may assume $U_j =\si_0^j$ is the first element in an orthonormal basis for $H^0(X_j,mL_j)$.

\v

Now choose a subset $A\sub \{(\al_0,...,\al_m)\,:\, 1\leq \al_j\leq N\}$ with $N_m+1$ elements such that
$\{s_\al^\i\,:\, \al\in A\}$ is a basis of $H^0(X_\i,mL_\i)$. This is possible since $\oplus_{m=0}^\i H^0(X_\i,mL)$ is generated by $H^0(X_\i,L)$.
\v
Write $U_\i$ as a linear combination of basis elements:

$$ U_\i\ = \ \s_{\al\in A} U_\i^\al\,s_\al^\i
$$
\v

with $U^\al_\i\in\R$. Then

$$\left<U_\i, s_\b^\i\right>_{{\rm Hilb}(h_\i^m)}\ = \ \s_{\al\in A}U_j^\al  \<s_\al^\i,s_\b^\i\>_{{\rm Hilb}(h_\i^m)}
$$

Note that  $\det\<s_\al^\i,s_\b^\i\>_{{\rm Hilb}(h_\i)} \ \not= 0$ since $\{s_\al^\i:\al\in A\}$ is a basis. Since $s_\al^j\ra s_\al^\i$ by virtue of (\ref{products}), we see that 
$\det\<s_\al^j,s_\b^j\>_{{\rm Hilb}(h_j)}  \not= 0$ for $j$ sufficiently large. Moreover, $U^\i$ is defined so that
$U^j\ra U^\i$. Cramer's rule now implies that $U_j^\al\ra U_\i^\al$. But this contradicts our assumption, namely 
$\max_\al \{|U_j^\al|\} \geq j$. \qed

\v\v
\section{Proof of Theorem \ref{m}}
\setcounter{equation}{0}
\v

After replacing $\cF(n,d)$ by $\cF_r(n,d)$ for some fixed highly divisible positive integer $r$, we may assume $m_0(n,d)=1$ (defined in section 1) and $r(n,S)=1$ (defined in 
Theorem \ref{KZ}). Also, without loss of generality we may assume $K=1$ so that all elements in $\cF(n,d)$ have the same Hilbert polynomial $p(x)$. In other words, if 
$(X,\o_X, L, h_L)\in\cF(n,d)$ we may assume
$$ \dim H^0(X, mL)\ = \ p(m)\ \ {\rm for\ all}\ \ m\geq 1
$$
\v
We let $N_m+1=p(m)$ and $N=N_1$. If $(X,\o_X, L, h_L)\in\cF(n,d)$ then $T_X: X\hookrightarrow\P^N$ is an embedding, where 
\be\label{TX} T_X(x)\ = \ (s_{0,X}(x), s_{1,X}(x),...,s_{N,X}(x))
\ee
\v

\v
Here $\{s_{0,X}, s_{1,X},...,s_{N,X}\}$ is an orthonormal basis of $H^0(X,L)$ with respect to 
${\rm Hilb}_{h_L}$. In this way we may view $X$ as a submanifold of $\P^N$, for all $(X,\o_X,L,h_L)\in\cF(n,d)$ so  $\t_X=\T_{\rm FS}\big|_X$ where $\T_{\rm FS}$ is the Fubini-Study metric on $\P^N$ and $L=O_{\P^N}(1)\big|_X$ and $h_X=H_{\rm FS}\big|_L$. With these identifications we have

\be\label{hx} h_X\ = \ {1\over |s_0|^2+\cdots + |s_N|^2}
\ee
\v 
as well as

$$\hbox{$h_L=h_Xe^{-\phi_X}$,\ \
$\Ric(h_X)=\t_X$ and
$ \Ric(h_L)= \o_X = \t_X+\sq \ddb \phi_X$.

}
$$
\v
Note that if $\{E_0,...,E_N\}$ is the  basis of 
$H^0(\P^N,O(1))$, given by $E_j(x_0,...,x_N)=x_j$, then we may write

\be\label{restr}  s_{j,X}\ = E_j\big|_X
\ee

\v

In other words, $E_j$ is a section of $H^0(\P^N,O(1))$, independent of $X$, and is simultaneously an extension of $s_{j,X}$ from $H^0(X,L)$ to $H^0(\P^N,O(1))$ for all $(X,L)\in\cF(n,d)$.
\v
Now we fix $(X,\o_X, L, h_L)\in \cF(n,d)$. For $m\geq 1$ we define the Bergman potential
\v
$$
\phi_m\ = \ {1\over m}\log\left(
\s_{i=0}^{N_m} |\si_i|_{h_X^m}^2
\right)
\ = \ 
{1\over m}\log\left(
\s_{i=0}^{N_m} |\si_i|_{h_L^m}^2\right)
\ + \ \phi_X
$$
\v
where $\{\si_0,...,\si_{N_m}\}$ is an orthonormal basis of $H^0(X, mL)$ and 
$h_L=h_Xe^{-\phi_X}$. Note that (\ref{hx}) implies $\phi_1=0$.
\v
Corollary \ref{pc} implies 
\v

\be\label{logA} |\phi_X-\phi_m|\ = \ {1\over m}|\log\r(X, m\o_X, mL, h_L^m)| \ \leq \ {\log A\over m}\ 
\ee

\v

In particular, if $(X,L)\in\cF(n,d)$ then

$$ |\phi_X|\ = \ |\phi_X-\phi_1|\ \leq \ \log A
$$
\v
We will also make use of the following. If $\si\in H^0(X,mL)$ then

\be\label{hilb} \|\si\|^2_{\rm Hilb{(h_X^m)}} \ \leq \  \|\si\|^2_{\rm Hilb{(h_L^m)}}\ \leq \ e^{mA} 
\|\si\|^2_{\rm Hilb{(h_X^m)}}
\ee

since  

$$
\|\si\|^2_{\rm Hilb(h_X^m)}\ = \ \I_X\si\bar\si h_X^m {(m\o)^n\over n!}\ = \ 
\I_X\si\bar\si h_L^m  e^{m\phi_X}{(m\o)^n\over n!}\ 
$$

and $-mA\leq m\phi_X\leq 0$.

\v
Fix $x\in X$ and assume that $\si_i(x)=0$ if $i>0$ so 
$\r_{h_L^m}(x)=|\si_0|^2_{h_L^m}$ (this can be achieved by a unitary transformation).
\v
Let $\na$ be the Chern connection on $L$ for $h_X$. The induced connection on $mL$ for $h_X^m$ will be denoted by $\na$ as well. Now we compute
\v
$$
\pl_p\phi_m(x)\ = 
{1\over m}\cdot {\s_{i=0}^{N_m} (\na_p\si_i)\bar \si_i h_X^m\over \s_{i=0}^{N_m} \si_i\bar\si_ih_X^m}(x)
\ = \  {1\over m}{(\na_p \si_0)\bar \si_0\,h_X^m\over
\si_0\bar\si_0\, h_X^m}(x)\ 
=
{1\over m}{(\na_p \si_0)\bar \si_0\,h_L^m\over
\si_0\bar\si_0\, h_L^m}(x)
$$
\be\label{phim}
\Lra \ 
|\na\phi_m|^2_{\t_X}(x)
\ = \ \t_X^{p\bar q}\pl_p\phi_m\,\pl_{\bar q}\phi_m\ = \ 
{1\over m^2}{\,
\big[\t_X^{p\bar q}(\na_p\si_0)(\na_{\bar q}\bar\si)h_L^m\big]
\over
\r_{h_L^m}
}(x)\ \leq \ {A\over m^2}|\na\si_0|^2_{\t_X,h_L^m}(x)
\ee
\v
where have made  use of  $\r_{_{h_L^m}}(x)=|\si_0|^2_{h_L^m}(x)$ in the last identity.
Here we have defined $|\na\si_0|^2_{\t_X,h_L^m}:=
\t_X^{p\bar q}(\na_p\si_0)(\na_{\bar q}\bar\si_0)h_L^m$. Now our goal is to find a good estimate for $|\na\si_0|^2_{\t_X,h_L^m}$.

\v\v

\begin{lemma}\label{lemma} Let $m\geq 1$. There exists $\m_m(n,d)>0$ with the following property. For every $(X,\o_X,L,h_L)\in\mathcal{F}(n,d)$ and for every $\si\in H^0(X,mL)$ we have
\be\label{stronger} |\si|^2_{h_L^m}\ \leq A\|\si\|^2_{\rm Hilb(h_L^m)}\ \ {\rm and} \ \ 
|\na\si|_{\t_X,h_L^m} \ \leq \ \m_m\|\si\|_{\rm Hilb(h_L^m)}
\ee
\end{lemma}
\v

{\it Proof.} The first estimate follows from the partial $C^0$ estimate. As for the second, we must show
that
\v
\be\label{weaker} \sup\{x\in X: \t_X^{p\bar q}(\na_p\si)(\na_{\bar q}\bar\si)h_L^m(x)\}\ \leq \ \m_m\I_X \sigma\bar\sigma h_L^m{(m\o)  ^n\over n!}
\ee

\v  It suffices to show that

\be\label{weaker1} \sup_X |\na\si|^2_{\t_X,h_X^m}=\sup\{x\in X: \t_X^{p\bar q}(\na_p\si)(\na_{\bar q}\bar\si)h_X^m\}\ \leq \ \m_m 
\|\si\|^2_{\rm Hilb(h_X^m)}
\ee

\v

making use of (\ref{hilb}) and $h_L^m=h_X^me^{-m\phi_X}\leq e^{mA}h_X$.

\v

We first  prove (\ref{weaker1}) in the case where $m=1$. Let
$(X,\o_X,L,h_L)\in\mathcal{F}(n,d)$ and let
 $s=\si\in H^0(X,L)$.
After scaling, we may assume $\|s\|_{\rm Hilb(h_L))}=1$. After transformation by a matrix in $U(N+1)$, we may assume that $s$ is the first element in the orthonormal basis that defines $T: X\ra \P^{N}$. In other words, $s$ is the restriction of $E_0$, the first of the $N+1$ coordinate functions on $\C^{N+1}$. Thus we can write (\ref{restr}) as

\be\label{restr1}  s\ = E_0\big|_X
\ee

\v

 Let $\Theta_{\rm FS}$ denote the Fubini-Study metric on $\P^N$ and $H_{\rm FS}$ the Fubini-Study metric on $O(1)$. Then $s, \t_X$ and $h_X$ are the restrictions of $E_0, \Theta_{\rm FS}$ and $H_{\rm FS}$ from $H^0(\P^N,O(1)), \P^N$ and $O(1)$ to $H^0(X,L),  X$ and $L$. Let $\ti\na$ be the connection on $O(1)$ for $H_{\rm FS}$. Then, for $x\in X$ and  $V\in T_xX\sub T_x\P^N$, we have 
 
 $$(\na_Vs)(x) = (\ti\na_V E_0)(x)
 $$

 Thus if $V_1,..,V_n$ is an orthonormal frame for $T_xX$ and $V_1,..,V_n, V_{n+1},...,V_N$  and orthonormal frame for $T_x\P^N$, we have
 $$ |\na\si|_{\t_X,h_X}^2\ = \ \s_{j=1}^n|\na_{V_j}\si|^2\ = \  \s_{j=1}^n|\ti\na_{V_j}E_0|^2\ \leq \ 
 \ \s_{j=1}^N|\na_{V_j}E_0|^2\  = \ |\ti\na E_0|^2_{\Theta_{\rm FS},H_{\rm FS}}(x)
 $$
 
So we see
\be\label{mu1}
|\na s|^2_{\t_X,h_L}(x)\ \leq \ |\na s|^2_{\t_X,h_X}(x)\ \leq \ |\ti\na  E_0|^2_{\Theta_{\rm FS},H_{\rm FS}}(x)
\ \leq\ \m_1:=\ \sup_{\P^N}\ |\ti\na  E_0|^2_{\Theta_{\rm FS},H_{\rm FS}}
\ee
\v
 since $|\ti\na  E_0|^2_{\Theta_{\rm FS},H_{\rm FS}}$ is a continuous function on the compact set  $\P^N$ and is thus bounded by a constant $\m_1=\m_1(N)$. 
The proves (\ref{weaker1}) in the case $m=1$.
\v
Now assume $m>1$ and let $\si\in H^0(X,mL)$. Proposition (\ref{weak}) implies
\be\label{2}
\si\ = \ \s_{\al_1,...,\al_m=1}^{N}\,\si({\al_1,...,\al_m})\,s_{\al_1}\cdots s_{\al_m}
\ee
\v
where 
$$ \si({\al_1,...,\al_m})\in \R,  \ 
$$
\v
\be\label{lamdam} |\si({\al_1,...,\al_m})| \ \leq \ \lambda_m
\|\si\|_{{\rm Hilb(h_L^m)}}
\ee

Thus

$$
\na\si\ = \ \s_{\al_1,...,\al_m=0}^{N}\,\si({\al_1,...,\al_m})\s_{j=1}^m\,s_{\al_1}\cdots \na s_{\al_j}\cdots s_{\al_m}\ \Lra \ 
$$

$$ |\na\si|_{\t_X,h_L^m}\ \leq \ (N+1)^m\lambda_m\,mA^{m-1}\m_1\|\si\|_{{\rm Hilb(h_L^m)}}
\ :=\ \m_m\,\|\si\|_{{\rm Hilb(h_L^m)}} 
$$
\v
where $\m_1$ is defined in (\ref{mu1}).
\v

\begin{lemma}\label{lemma} There exists $B>0$ with the following property. For all $(X,\o,L,h_L)\in\mathcal{F}(n,d)$ and for all $\si\in H^0(X,mL)$ we have
\be\label{stronger}  
|\na\si|_{\t_X,h_L^m} \ \leq \ B^m\|\si\|_{\rm Hilb(h_L^m)}
\ee
\end{lemma}
\v
{\it Proof.} We must show
\be\label{stronger} \sup\{x\in X: \t_X^{p\bar q}(\na_p\si)(\na_{\bar q}\bar\si)h_L^m(x)\}\ \leq \ B^{{2m}}\I_X \sigma\bar\sigma h_L^m{(m\o^n)\over n!}
\ee

Proposition \ref{weak} gives us a the following weaker statement. For each $m\geq 1$ there exists $\m_m>0$ such that
\v
\be\label{weaker} \sup\{x\in X: \t_X^{p\bar q}(\na_p\si)(\na_{\bar q}\bar\si)h_L^m(x)\}\ \leq \ \m_m\I_X \sigma\bar\sigma h_L^m{(m\o)^n\over n!}
\ee

\v i.e.,
$|\na\si|^2_{\t_X,h_L^m} \ \leq \ \m_m\|\si\|_{\rm Hilb(h_L^m)}$ (this is weaker since we must further show $\m_m=B^m$). 
In order to prove (\ref{stronger}) we apply Theorem \ref{L} with

$$ k=1,\ \  l=m-n-2,\ \ \ {\rm so \ that} \ \  m-lk=n+2
$$

\v
After scaling, we may assume $\|\si\|_{\rm Hilb(h_L^m)}=1$ so, 
we obtain 

\v
$$
\si\ = \ \s_{\al_1,...,\al_l=0}^{N_1}\,U({\al_1,...,\al_l})\,s_{\al_1}\cdots s_{\al_l}
$$

\v

$$ \|U({\al_0,...,\al_l})\|^2_{
{\rm Hilb}(h_L^{n+2})
}\  \ \leq \
{(m-2)!\over n!(m-n-2)!}\cdot {(n+2)^n\over m^n}A^{2n+2+l}
\|\si\|^2_{
{\rm Hilb}(h_L^{m})
}\ 
$$
$$\leq c_1(n,d)A^m\,\|\si\|_{
{\rm Hilb}(h_L^{m})
}^2=c_1(n,d)A^m
$$

\v
where $c_1(n,d)={(n+2)^n}A^n$.
Now  

\v

$$\na\si\ = \ \s_{\al_1,...,\al_l=1}^{N_1}
(\na U({\al_1,...,\al_l}))
s_{\al_1}\cdots  s_{\al_l}\ 
+ \ 
\s_{\al_1,...,\al_l=1}^{N_1}U({\al_1,...,\al_l})
\s_{p=1}^ls_{\al_1}\cdots \na s_{\al_p}\cdots s_{\al_l}
$$

\v
Noting that $U(\al_1,...,\al_l)\in H^0((n+2)L)$ and $|s_{\al_j}|^2\leq A$ and applying (\ref{weaker}) we have
$$|\na\si|_{\t_X,h_L^m}\ \leq \ 
\s_{\al_1,...,\al_l=1}^{N_1} \m_{n+2}[c_1(n,d)A^m]\cdot  A^l
\ +\ 
\s_{\al_1,...,\al_l=1}^{N_1}A[c_1(n,d)A^m]\cdot 
\s_{p=1}^l A^{l-1}\m_1
$$
$$\ = \ 
 N_1^l\left[\m_{n+1}[c_1(n,d)A^m]\cdot  A^l
\ +\ 
A[c_1(n,d)A^m]\cdot lA^{l-1}\m_1\right]\ \leq \ B^m
$$
\v
Where $B = N_1\m_{n+1}c_1(n,d)A^3\m_1$. Here we use the inequality $A^mlA^l\leq A^mA^mA^m$. This proves Lemma \ref{lemma}.
\v

\v
\begin{lemma} Let $n,d\geq 1$. Then there exists $C=C(n,d)>0$ and $\k=\k(n,d)>0$  with the following properties. Let $r\geq 2$ and $m=(n+2)^r$.
Let $U\in H^0(X,mL)$.  Then

$$  |\na U|_{\t_X,h_L^m}\ \leq \ Cm^\k\|U\|_{\rm Hilb(h_L^m)}
$$
\end{lemma}
\v
{\it Proof.}
Let  
\v

$$   k=(n+2)^{r-2},\ \ l=(n^2+3n+2) \ \ \Lra
 \ \ m-kl\ = \ (n+2)^{r-1}=(n+2)k
$$
\v
Thus $m=(n+2)k+kl=(n+2+l)k\geq (n+2+l)k$ so we may apply Theorem \ref{L}.
\v
Fix $x\in X$ and choose the orthonormal basis $\{s_0,...,s_{N_k}\}$ of $H^0(X,kL)$ such that $s_0(x)\not=0$, $s_1,...,s_n$ vanishing at $x$ to first order and the remaining sections vanishing to at least second order. Then Theorem \ref{L} implies
\

\be\label{u} U \ = \ \s_{p=0}^n u_p\,s_{0}^{l-1}\,s_{p} \ + \ O(x^2)
\ee
\v
where

$$ \|u_p\|^2_{(n+2)^{r-1}}\ \leq \ \left({n^2+4n+2\atop n}\right)
A^{2n+2+n^2+3n+2}\|U\|^2_{(n+1)^r}\ := \ q(n,d)^2\cdot \|U\|^2_{(n+2)^r}
$$

\v

Applying $\na$ to (\ref{u}) and evaluating at $x$, we get

\v

\be\label{nu}\na U\ = \  (\na u_0) s_{0}^{l}
\ + \ 
\s_{p=0}^n\,u_p\, l_p^*s_{0}^{l-1}\na s_{p}
\ee

where $l_p^*=l$ if $p=0$ and $l_p^*=1$ otherwise. Thus

\v
$$
|\na U|_{\t_X,(n+2)^r}
\ \leq \ 
|\na u_0|_{\t_X,(n+2)^{r-1}} |s_{0}|_{(n+2)^{r-2}}^{l}
\ + \ 
\s_{p=0}^n\,|u_p|_{(n+2)^{r-1}}\, l\,|s_{0}|_{(n+2)^{r-2}}^{l-1}|\na s_{p}|_{\t_X, (n+2)^{r-2}}
$$

\v

$$\leq B^{(n+2)^{r-1}}\|u_0\|_{(n+2)^{r-1}}A^l
\ + \ 
(n+1)A
\|u_p\|_{(n+2)^{r-1}}lA^{l-1}B^{(n+2)^{r-2}}
$$

\v
$$
\leq \ 
q(n,d)\cdot \|u\|_{(n+2)^{r}}\left(B^{(n+2)^{r-1}}A^l
\ + \ 
(n+1)lA^{l}B^{(n+2)^{r-2}} 
\right)
$$
\v
$$
\leq \ 
q(n,d)\|u\|_{(n+2)^{r}}\,B^{(n+2)^{r-1}}A^{l}(n+1)l
 := \ D\cdot B^{(n+2)^{r-1}}\|u\|_{(n+2)^{r}}
$$

\v

where

$$
D\ = \ q(n,d) 
A^{l}(n+1)l
$$
\v

 We claim that for all $r\geq 1$ and all 
$u\in H^0(X,(n+1)^rL)$ that

\be
\label{ind}
|\na U|_{\t_X,h_L^m}\ \leq \ D^tB^{(n+2)^{r-t}}\|U\|_{(n+2)^r}\ \ 1\leq t\leq r-2
\ee

We proceed by induction on $r$ and $t$. We know (\ref{ind}) is true for $t=1$ by the previous estimate. Now assume it holds for $t$ and try to show it holds for $t+1$. 
\v
Recall that $u_p\in H^0((n+2)^{r-1}L)$ and 
$s_{\al_{\al_j}}\in H^0((n+2)^{r-2}L)$.
Thus if we apply (\ref{ind}) to $u_p$ and $s_{\al_j}$ in (\ref{nu}) we obtain

\v
$$\na U\ = \  (\na u_0) s_{0}^{l}
\ + \ 
\s_{p=0}^n\,u_p\, l_p^*s_{0}^{l-1}\na s_{p}
$$
\v
$$|\na u|\ \leq \ D^tB^{(n+2)^{r-1-t}} \|u_0\|_{(n+2)^{r-1}}A^l
\ + \ 
(n+1)A\|u_p\|_{(n+2)^{r-1}}lA^{l-1}D^tB^{(n+2)^{r-2-t}}
$$
\v
$$
|\na u|\ \leq \ D^tB^{(n+2)^{r-1-t}} 
\left(q\cdot \|u\|_{(n+2)^r}\right)A^l\ + \
(n+1)A\,\left(q\cdot \|u\|_{(n+2)^r}\right)lA^{l-1}D^tB^{(n+1)^{r-2-t}}
$$

\v

$$ \ 
\leq \ D^tB^{(n+2)^{r-1-t}} q
\|u\|_{(n+2)^r}A^{l}(n+1)l
\ = \ 
DD^tB^{(n+2)^{r-1-t}} 
\|u\|_{(n+2)^r}\ 
$$

\v
$$\ = \ D^{t+1}B^{(n+2)^{r-1-t}}\|u\|_{(n+2)^r}\ 
$$
\v

Letting $t=r-2$ we get
\v
$$ |\na u|_{\t_X,h_L^m}\ \leq \ D^{r-1}B^{n+2}\|u\|_{(n+2)^r}\leq D^rB^{n+2}\|u\|_{(n+2)^r}\
$$
\v
Since $m=(n+2)^r$ we see $r={\log m\over \log(n+2)}$ so
\v
$$|\na u|_{\t_X,h_L^m}\ \leq \ B^{n+2} P^{\log m}\|u\|_{(n+2)^r}\\ = \ C_1m^{\log P}  =   C_1m^\kappa\|u\|_{(n+2)^r}$$
\v
where $P(n,d)=D^{1\over\log(n+2)}$ and $\k=\log P$. Applying this to (\ref{phim}) we obtain

$$ |\na\phi_m|_{\t_X} \ \leq \ {\sqrt A\over m}C_1m^\kappa
\ \leq \ C_2\,m^{\kappa-1}
$$
\v
Now we can complete the proof of Theorem \ref{m}. Let $x,y\in X$. Then
\v
$$ |\phi_X(x)-\phi_X(y)|\ \leq |\phi_X(x)-\phi_m(y)|+|\phi_m(x)-\phi_m(y)|+|\phi_m(y)-\phi_X(y)|
$$
\v
$$\leq\ {2\log A\over m}\ + \ \sup_X|\na\phi_m|_{\t_X}\cdot d_{\t_X}(x,y)\ \leq \ 
{C_0\over m}\ + \ C_1m^{\k-1} d_{\t_X}(x,y)
$$
\v

\v
where $C_0={2\log A}$. Choosing $r\geq 2$ such that
\v
$$ (n+1)^r\ \leq  \ \ {1\over 
[d_{\t_X}(x,y)]^{1\over\kappa}}\ <  \ (n+1)^{r+1}
$$
\v
and setting $\al={1\over\kappa}$ we see that
\v
$$ {1\over (n+1)^{r+1}}\ \leq \ d^{1\over\k}\ \ \Lra \ 
{1\over (n+1)^{r}}\ \leq \ d^{{1\over\k}{r\over r+1}}\  
\leq \ d^{2\over 3\k}
$$

\v
$$|\phi(x)-\phi(y)|\ \leq \ 
C\,[d_{\t_X}(x,y)]^\al
$$

\v

$$|\phi(x)-\phi(y)|\ \leq \ 
C_0\,[d_{\t_X}(x,y)]^{1\over\k}
$$
\v
where $C=C_0+C_1$ and
$$\al\ = \ {\log(n+1)\over \log D}
$$
\v
\section{Proof of Theorem \ref{H}}
\setcounter{equation}{0}
\v

\v
In order to deduce Theorem \ref{H} from Theorem \ref{m},  we shall express $(X_\i,\o_\i):=(X,\o_{\rm KE})$  as the limit of smooth \K manifolds  $(X_j,\o_j)$, apply   (\ref{holder}) to each $X_j$, and take the limit as $j\to\i$ in order to obtain (\ref{holderke}). More precisely, if $x_\i,y_\i\in X$, we write $x_\i=\lim_{j\to\i}x_j$  and  $y_\i=\lim_{j\to\i}y_j$.  If $\o_{\rm FS}+\sq\ddb\phi_j$ is the \K metric on $X_j$, and $d_{X_j}$ the intrinsic Fubini-Study metric, then (\ref{holder}) yields
\v

\be \label{51}|\phi_{X_j}(x_j)-\phi_{X_j}(y_j)|\leq Cd_{X_j}(x_j,y_j)^{\al}
\ee

\v
If $j\ra \i$, then, by the smoothability assumption, the left side of (\ref{51}) converges to $|\phi_{X}(x)-\phi_{X}(y)|$, which is the left side of (\ref{holderke}). What remains is to show that 

\v
\be\label{liminf}
\limsup_{j\to\i}d_{X_j}(x_j,y_j)\leq Cd_{X_\i}(x_\i,y_\i)^{(\m/\al)}
\ee
for some $C,\m$, independent of $x,y$,
 and this is the content of Proposition \ref{fs} below.
\v
Before stating Proposition \ref{fs} we need  first to recall some basic defnitions.

\v
If  $\g:[0,1]\ra \P^N$ is piecewise $C^1$ curve, we define
\v

$$ \ell_{FS}(\g)\ = \ \I_0^1\, |\g'(t)|_{\rm FS}\, dt
$$
\v
Here $\g'(t)\in T_{\g(t)}(\P^N)$ and $|\g'(t)|_{\rm FS}$ its length with respect to the Fubini-Study metric.

\v

Let $X\sub\P^N$ be a normal variety and let $x,y\in X$. We define $d_{\rm FS}(x,y)$ the extrinsic Fubini-Study distance, and $d_{X}(x,y)$,  the intrinsic Fubini-Study distance, as follows.  

\v
$$ \hbox{$d_{\rm FS}(x,y)\ = \ \inf\{\ell_{\rm FS}(\g)\, | \, \g: [0,1]\ra\P^N,\ $\ { $\g(0)=x$, $\g(1)=y$\, \} }\   
}
$$
\v
$$ \hbox{$d_{X}(x,y)\ = \ \inf\{\ell_{\rm FS}(\g)\, | \, \g: [0,1]\ra\P^N,\ $\ $\g(0)=x$, $\g(1)=y$,  $\g(t)\in X\ {\rm for\ all}\ t\}$
}
$$
\v
Here we require all curves $\g$ to be $C^1$.
\v

%****************************************************************************

Next we  recall the definition of algebraic convergence.
Let $X_1, X_2,...\sub\P^N$ be a sequence of smooth projective manifolds with the same Hilbert polynomial $P\in\Q[x]$.  This means that for each $j$, we have
$\dim(H^0(X_j,O(k))=P(k):=N_k+1$ for $k\geq k_0$. One can show that $k_0$ is independent of $j$.
By compactness of the Hilbert scheme, after passing to a subsequence we have  algebraic convergence $X_j\ra X_\i$. Here $X_\i\sub\P^{N_k}$ is a projective variety (possibly singular) with Hilbert polynomial P. To say that the convergence is algebraic is equivalent to saying that the corresponding sequence of points in the Hilbert scheme converge. More concretely, for each $1\leq j\leq\i$, there exists $m,d>0$, and $f_j^{(1)}, f_j^{(2)}, ...f_j^{(m)}\in\C[T_0,...,T_N]$, homogeneous of degree $d$, 
  such that 
$X_j=\{x\in\P^N\,:\, f_j^{(k)}(x)=0,\ 1\leq k\leq m\}$ and satisfying the following property. For each $k$, the coefficients of $f_j^{(k)}$ converge to those of $f_\i^{(k)}$. This implies that for every $\e>0$ there exists $j_0>0$ such that for $j\geq j_0$, the variety $X_j$ lies in an $\e$-neighborhood of $X_\i$. Thus, for every $x_\i\in X_\i$ there exist $x_j\in X_j$ such that $x_j\ra x$ in the topology of $\P^N$.
\v

\begin{proposition}\label{fs} Fix $k\geq k_0$ and let $X_j, X_\i\sub\P^{N_k}$, be as above. Assume $X_j\ra X_\i$ algebraically and that $X_\i$ is a normal variety. Let $x_j,y_j\in X_j$ and assume $x_j\ra x_\i\in X_\i$ and $y_j\ra y_\i\in X_\i$. Then there are $C,\m>0$ such that

$$\label{liminf}
\limsup_{j\to\i}d_{ X_j}(x_j,y_j)\ \leq \ C\,d_{X_\i}(x_\i,y_\i)^\m
$$

\end{proposition}
\v
The method of proof is as follows. First we use a theorem of Łojasiewicz to show there exists $C,\m>0$ with the following property. For any $x_\i,y_\i\in X_\i$, and for every $\e>0$, there is a piecewise analytic arc $\g_\i: [0,1]\ra X_\i$ joining $x_\i$ to $y_\i$, such that $\g_\i(t)\in X^{\rm reg}$ for all $t\in  [0,1]$ and $\ell_{X_\i}(\g_\i)\leq Cd_{\rm FS}(x_\i,y_\i)^\m_{\rm FS}+\e$. Then we use \cite{DS} to construct curves $\g_j:[0,1]\ra X_j$ joining $x_j$ to $y_j$ such that $\ell_{X_j}(\g_j)\leq \ell_{X_\i}(\g_\i)+\e$ for $j$ sufficiently large. Thus 
\v
$$d_{X_j}(x_j,y_j)\leq \ell_{X_j}(\g_j)\leq  \ Cd_{X_\i}(x_\i,y_\i)^\m+\e$$
\v
 Here, by abuse of notation, we write $\m$ in place of $\m/\al$.
 
 \v
\begin{definition}
Let $M$ be a real analytic manifold and $E\sub M$ a subset. We say that $E$ is semianalytic if for every $p\in E$ there is an open set $p\in V\sub M$ such that

$$ E\cap V\ = \ \bigcup_{i=1}^p\left(
\bigcap_{j=1}^q \{g_{ij}>0\}\cap\{f_i=0\}
\right)
$$

where $g_{ij},f_i$ are real analytic functions on $V$.
\v
A function $f:M\ra\R$ is semianalytic if its graph in $M\times\R$ is semianalytic.
\v
%A continuous function $\eta:[0,1]\ra M$ is called an analytic arc if its restriction to $(0m1)$ is analytic.
\end{definition} 
 
{\it Proof of Proposition \ref{fs}.}
Since $X:=X_\i$ a connected analytic set (and hence a connected semianalytic set), we can apply Łojasiewicz's theorem \cite{Lo} (see also Theorem 6.10 of \cite{BM}), and conclude that there exist $C,\m>0$ with the following property. For every $x,y\in X$, there is a semianalytic arc $\eta: [0,1]\ra X$ joining $x$ to $y$, such that  $\ell_{X}(\eta)\leq Cd_{\rm FS}(x,y)^\m$.

\v

Let $\pi: \ti X\ra X$ be a resolution of singularities, and
fix an embedding of $\ti X$ into projective space. Let $\ti\o_{\rm FS}$ be the restriction of the Fubini-Study metric to 
$\ti X$. 

\v

\begin{lemma}\label{four} 
There exist $\ti x, \ti y, \in \ti X$ and a  piecewise analytic curve 
$\ti\eta$ joining $\ti x$ to $\ti y$ with the following properties.
\begin{enumerate}
\item $\ti\eta$ has finite length $L$.
\item $\pi(\ti x)=x$ and $\pi(\ti y)=y$.
\item
$\ell_{X}(\pi\circ \ti\eta)\ \leq \ 
\ell_X(\eta)\ \leq \ Cd_{\rm FS}(x,y)^\m$
\end{enumerate}
\end{lemma}

\v

{\it Proof of Lemma \ref{four}}. Since $[0,1]$ is one dimensional,
Theorem 1 of section 23 in \cite{Lo} implies that $C\sub X$,
the image of $\eta$, is a semianalytic set. 

\v

 Let
$\ti Z= \pi^{-1}(C)\sub \ti X$. Then $\ti Z\sub\ti X$ is
a semianalytic set since it is the inverse image of a semianalytic set. Since $X$ is normal and $C$ is connected, we may apply
Zariski's connectedness theorem to conclude
 $\ti Z$ is connected.
 \v
Fix
$\ti x^*\in\pi^{-1}(x)$ and $\ti y^*\in\pi^{-1}(y)$. Since $\ti Z$ is a connected semianalytic set, we my apply
Łojasiewicz's theorem again to obtain a semianalytic arc 
$\ti\eta^*: [0,1]\ra \ti Z$ of finite length joining $\ti x^*$ to $\ti y^*$.
\v
Let $\ti x = \ti\eta^*(t_1)$ where 
$t_1=\sup\{t\in [0,1]\,:\, \ti\eta^*(t)\in\pi^{-1}(x)\}$. Similarly, let $\ti y = \ti\eta^*(t_2)$ where 
$t_2=\inf\{t\in [0,1]\,:\, \ti\eta^*(t)\in\pi^{-1}(y)\}$.
\v
{\bf Case 1}. The graph $\ti\Gamma^*=\{(t,\ti\xi)\in [0,1]\times \ti X\,:\, \ti\xi=\ti\eta^*(t)\}$ is contained in an analytic curve 
$\ti K\sub \R\times\ti X$ which is smooth at all points for which $0<t<1$.

\v
Let $a_0=0$ and $b_0=\sup \{b\in [0,1]\,:\, \ti\eta^*:[a_0,b]\ra \ti X$ is injective $\}$. 
\v
Let $a_1=\sup\{a\in [b_0,1]\,:\, \ti\eta^*(a)=\ti\eta^*(b_0)\}$.
\v
Let $b_1=\sup \{b\in [a_1,1]\,:\, \ti\eta^*:[a_1,b]\ra \ti X$ is injective $\}$.
\v
This process must terminate after a finite number of steps since 
$\ti\Gamma^*$ is contained in an analytic curve. If $m$ is the number of steps we obtain a sequence

$$0=a_0<b_0<a_1<b_1<\cdots < a_m<b_m=1
$$

Now we define $\ti\eta$ as follows. If $t\in [a_j,b_j]$ then
$\ti\eta(t)=\ti\eta^*(t)$. Since all fibers of $\pi$ are connected, we may define
$\ti\eta:[b_j,a_{j+1}]\ra \ \pi^{-1}(\pi(b_j))$ to be any curve of finite length connecting $\ti\eta^*(b_j)$ to 
$\ti \eta^*(a_{j+1})$. Then 
$\ell_{X}(\pi\circ \ti\eta)\ \leq \ 
\ell_X(\eta)  
$.

\v
{\bf Case 2}. The graph $\ti\Gamma^*$ is contained in an analytic curve 
$\ti K\sub \R\times\ti X$.
\v
Let $\Sigma\sub\ti K$ be the singular set and 
$0\leq t_0 < t_1<\cdots< t_r\leq 1$ the projection of $\Sigma$ to $[0,1]$. We let $x_j=\ti\eta^*(t_{j-1})$ and $y_j=\ti\eta^*(t_j)$.
Then we can apply
Case 1 to obtain continuous curves $\ti\eta_j:[t_{j-1},t_j]\ra M$
which are analytic on $(t_{j-1},t_j)$ such that
$\s_j\ell_{X}(\pi\circ\ti\eta_j)\ \leq \ \ell_X(\eta)$.
To define $\ti\eta$ we join $\ti\eta_j(t_j)$ to $\ti\eta_{j+1}(t_j)$ by a piecewise analytic curve that lies in their common fiber so that

$$ \ell_X(\pi\circ\ti\eta)\ = \ \s_j\ell_{X}(\pi\circ\ti\eta_j)\ \leq \ \ell_X(\eta)
$$
\v

In general, $\ti\eta^*:[0,1]\ra \ti X$ is semianalytic.
Since the graph $\ti\Gamma^*$ of $\ti\eta^*$ is a one dimensional semianalytic set, we may apply the stratification theorem of \cite{Lo} (see also Theorem 2.13 of \cite{BM}) to conclude that 
$\Gamma$ is contained in an analytic curve $\ti K\sub \R\times \ti X$. Thus we are reduced to Case 2, and this completes the proof of the lemma. \qed

\v\v

{\it Continuation of the proof of Proposition \ref{fs}.} Write $I=\cup_{j=1}^A I_j$ where $I_j=[t_{j-1},t_j]$ is such that $\ti\eta$ is continuous on 
$[t_{j-1},t_j]$ and analytic on $(t_{j-1},t_j)$, and $\ti\eta([t_{j-1},t_j])\sub V_j$. Here $V_j$ is a coordinate neighborhood with coordinate function $\phi_j: V_j\ra U_j\sub \C^n$ and $U_j$ is chosen so that
the exceptional divisor intersected with $V_j$ is contained in the set $\phi_j^{-1}(\{z_1z_2\cdots z_n=0\})$.

\v

%%%%%%%%%%%%%%%%%%%%%%%%%%%%%%%%%%%%%%%%%%%%%%%%%%%

%!!!!!!!!!!!!!!!!!!!!!!!!!!!!!!!!!!!!!!!!!!!!!!!!!!!!
\v

\v
Let $L=\ell_{\ti X}(\ti\eta)$ and $g_j$ the euclidean metric on $U_j$.
Choose $\e_1,\e_2,\e_3<{\e\over A}$ and $\d>0$ such that

$$ \big|\phi_j(\ti\eta(t_j-\d))-\phi_j(\ti\eta(t_j+\d))\big|\ <\ {\e_1\over A}
$$

$$ \big|\phi_{j+1}(\ti\eta(t_j-\d))-\phi_{j+1}(\ti\eta(t_j+\d))\big|\ <\ {\e_1\over A}
$$

$$ \ell_{g_j}\big(\phi_j\circ\ti\eta\big|_{[t_j-\d,t_j+\d]}\big)\ < \ {\e_2\over A}
$$

$$ |g_{pq}(x+\r)-g_{pq}(x)|\  <  \ {\e_3\over n^2 AL}
$$
\v
for all $x\in U_j$ and for all $\r\in\C^n$ with $|\r|<\d$.
\v

\v
%Modify $\ti\eta$ so that it is injective on $[t_{j-1}+\d, t_j-\d]=\cup_{m=1}^{B_j}I_j^m$. This is possible since
%$\ti\eta$ is analytic on the larger open interval $(t_{j-1},t_j)$. The modified $\ti\eta$ is analytic on $I_j^b$.

Define $\ti\g$ on the interval $[t_{j-1}+\d, t_j-\d]$ by the formula

$$ \phi_j\circ\ti\g_j\ = \ \phi_j\circ\ti\eta_j\ + \ (\r_1,...,\r_n)
$$ 
\v

and define it on $[t_j-\d,t_j+\d]$ so that it avoids the coordinate hyperplanes and has length at most ${\e_2\over A}$.
Here $\r\in\C^n$ is chosen so that $|\r|<\d$ and $\ti\g_j$ does not intersect the coordinate axes.

\v
Define $\g:=\, \pi\circ\ti\g$. Then

\v

\v

$$
\big|\,\ell_{X}(\g)-\ell_{X}(\eta)\,\big|
\ = \ 
\big|\,\ell_{\pi^*(\o_{X})}(\ti \g)-\ell_{\pi^*(\o_{X})}(\ti \eta)\,\big|
$$
$$ \leq \ \s_{j=1}^A \big|\,\ell_{\pi^*(\o_{X})}(\ti \g^j)-\ell_{\pi^*(\o_{X})}(\ti \eta^j)\,\big|
$$
\v
where $\g^j$ and $\eta^j$ are the restrictions to $I_j$. If we consider the restriction to 
$[t_{j-1}+\d, t_j-\d]$, each term on the right is bounded by 

$$ {\e_3\over n^2AL}\cdot n^2\cdot L\ = \ {\e_3\over A} 
$$ 
\v
The length over $[t_j-d,t_j+\d]$ is bounded by
$${\e_1\over A}+{\e_2\over A}
$$
\v
so the sum from $j=1$ to $j=A$ is bounded by $\e_1+\e_2+\e_3<\e$.
\v

%*********************************************************************
\v
{\it Proof of Proposition \ref{fs}.}
 Let $x_\i,y_\i\in X$ and choose $x_j,y_j\in X_j$ such that $x_j\ra x_\i$ and $y_j\ra y_\i$. Choose $\g_\i:[0,1]\ra X_\i^{\rm reg}$ whose length is at most $Cd_{\rm FS}(x_\i,y_\i)^\m$, and let $Z_\i\sub X_\i$ be its image. We want to show that there exist curves $\g_{j}:[0,1]\ra X_j$ joining 
$x_j$ to $y_j$ such that $\ell_{X_j}(\g_{j})\leq \ Cd_{\rm FS}(x_j,y_j)^\m+2\e$ for $j$ sufficiently large.
\v
Let $V_\i\sub X_\i^{\rm reg}$ be an small open set and let $\O\sub\P^N$ be an open set such that $\O\cap X_\i=V_\i$. There exist homogeneous polynomials
$f^{(\i)}_1,...,f^{(\i)}_{N-n}$ such that, for $z_\i\in \O$ we have $z_\i\in V_\i$ if and only if
$f^{(\i)}_m(z_\i)=0$ for $1\leq m\leq N-n$. The implicit function theorem implies that after possibly shrinking $V_\i$, reordering the variables and working in affine coordinates so that $V_\i\sub\C^N$, there is an open  $U\sub\C^n$ and a biholomorphic map $\psi^\i: U \ra V_\i \sub\C^N$ of the form 
$(z_1,...,z_n)\mapsto (z_1,...,z_n, \psi^{\i}_{n+1}(z),...,\psi^{\i}_N(z))$. Here $\psi^{\i}_m: \C^n\ra \C$ is a holomorphic function and $\phi^\i: V_\i\ra U$ is given by the projection map. In other words, $V_\i$ is a graph over the domain $U$. Similarly, \ we construct $\psi^{j}:U\ra V_j\sub X_j$ where $V_j\ra V_\i$.
If we choose $f_m^j$ for $1\leq m\leq N-n$ in such a way that $f_m^j\ra f_m^{(\i)}$,
 the proof of the implicit function theorem implies $\psi_\al^{(j)} \ra \psi_\al^{\i}$ on compact sets. In particular, $[{\psi^j}]^*\o_{\rm FS}\ra [\psi^{\i}]^*\o_{\rm FS}$ smoothly. If $V_\i$ covers $\g_\i([0,1])$ we define $\g_j=\psi^{j}\circ \phi^{\i}\circ\g_\i$, and we obtain $\g_{j}$ joining $x_j$ to $y_j$ and satisfying
 $d_{\rm FS}(x_j,y_j)\leq \ell(\g_{j})\ra \ell(\g_\i) \leq d_{\rm FS}(x_\i,y_\i)+\e$. 
 If $V_\i$ does not cover $\g_\i([0,1])$, then we choose a finite cover and construct $\g_j$ using a partition of unity.
 Hence
 \be\label{limsup}\hskip 1in\limsup_{j\to\i}\,  d_{\rm FS}(x_j,y_j)\leq Cd_{FS}(x,y)^\m\hskip 1in
 \qed
 \ee
 
\v

{\it Proof of Theorem \ref{H}}.  Let $(X,\o_{\rm KE})$ be a \KE variety and  $(X_i,\o_i)$  smooth \K manifolds such that $(X_i,\o_i)\xrightarrow{\text{GH}} (X,\o_{\rm KE})$. We write $\o_i=\o_{\rm FS}+\sq\ddb\phi_i$
and $\o_{\rm KE}=\o_{\rm FS}+\sq\ddb\phi_{\rm KE}$ with $\sup_{X_i}\phi_i=\sup_X\phi=0$.   Theorem \ref{m} implies that the $\phi_i$ are uniformly H\"older, so Proposition \ref{fs} implies that, after passing to a subsequence, they converge to $\phi_\i$ satsifying

$$
|\phi_{{\i}}(x)-\phi_{{\i}}(y)|\leq Cd_{\rm X}(x,y)^{\m}
$$
\v
for all $x,y\in X$.
On the other hand, 
We know from \cite{DS} that
$\o_i\ra\o_{\rm KE}$ smoothly on $X^{\rm reg}$. This implies $\phi_\i=\phi_{\rm KE}$ and completes the proof of Theorem \ref{H}.
\v

\v

\section{Proof of Theorem \ref{mainthm3}}

Suppose an $n$-dimensional smoothable projective variety $X$ with klt singularities in the central fibre $X=\mathcal{X}_0$ in the flat family
$$\pi: \mathcal{X} \subset \mathbb{P}^N \ra \D. $$
As in Theorem \ref{mainthm3}, we assume that $K_{\mathcal{X}}$ is $\mathbb{Q}$-Gorenstein and $K_{\mathcal{X}/\D}$ is $\mathbb{Q}$-Cartier. 
By inversion of adjunction,  $\mathcal{X}$ is locally klt and we choose an adaptive volume measure  $\Omega_\mathcal{X}$ on $\mathcal{X}$ after possibly shrinking $\D$. Let $\Omega_t$ be the relative volume measure defined by 
$$\Omega_t \wedge dt \wedge d \overline{t}  = \Omega_{\mathcal{X}} $$
for $t\in \D$. We can view $\Omega_t$ as an adaptive volume measure on $\mathcal{X}_t$ for each $t\in \D$ and let 
$$\Omega= \Omega_0.$$
 Let $\theta_t$ be the restriction of the Fubini-Study metric $\Theta$ on $\mathbb{P}^N$ restricted to $\mathcal{X}_t$.

\v

We consider the following complex Monge-Amp\`ere equation on $X$
\begin{equation}\label{orma}
(\theta_0 + \ddb \phi)^n = e^F \Omega, ~\sup_X \phi = 0. 
\end{equation}
where $F$ satisfies
$$A\theta_0 + \ddb F\geq 0, ~\int_X e^F\Omega = [\theta_0]^n. $$
for some $A>0$.
Equation (\ref{orma}) admits a unique bounded solution. The goal of this section is to prove that the solution $\phi$ to equation (\ref{orma}) is H\"older continuous with respect to $d_{\mathbb{P}^N}$. This will be achieved by regularizing the $\phi$ via the smoothing $\pi: \mathcal{X}\ra \D$ of $X$.

\v

We begin our proof with the following lemma due to \cite{GGZ1, GGZ2} (c.f. Lemma 3.5 \cite{GGZ2}).

\begin{lemma} \label{lpbd1} There exists $p>1$ and $C>0$ such that for all $t\in \D$, 
$$\left\| \frac{\Omega_t}{\theta_t^n} \right\|_{L^p(\mathcal{X}_t, \textcolor{red}{\Omega_t})} \leq C. $$

\end{lemma}

Since $\Omega_{\mathcal{X}}$ is a smooth adaptive volume measure on $\mathcal{X}$, there exists $B>0$ such that 
$$- B \Theta  \leq Ric(\Omega_\mathcal{X}) = - \ddb \log \Omega_\mathcal{X}  \leq B \Theta $$
on $\mathcal{X}$ and for all $t\in \D$, 
$$- B \theta_t \leq Ric(\Omega_t) = - \ddb \log \Omega_t \leq B \theta_t $$
on $\mathcal{X}_t$.

 \v

The following extension result for $F$ follows directly from the global extension theorem of \cite{CG} for quasi-plurisubharmonic functions on a normal projective variety. 

\begin{lemma} There exists $- \mathcal{F} \in PSH(\mathbb{P}^N, A\Theta)$ such that 
$$\mathcal{F}|_{X} = F. $$

\end{lemma}

We can immediately approximate $\mathcal{F}$ by the standard regularization for quasi-plurisubharmonic functions. 

\begin{lemma} There exist a sequence of $-\mathcal{F}_j \in C^\infty(\mathbb{P}^N)\cap PSH(\mathbb{P}^N, A\Theta)$ such that $\mathcal{F}_j$ decreasingly converge to $\mathcal{F}$.

\end{lemma} 

We let $$F_{t, j} = \mathcal{F}_j|_{\mathcal{X}_t} $$ for $t\in \D$ and consider the following complex Monge-Amp\`ere equation on $\mathcal{X}_t$
\begin{equation}\label{appma}
(\theta_t + \phi_{t, j} )^n = e^{F_{t, j} + c_{t, j}} \Omega_t, ~\sup_{\mathcal{X}_t} \phi_{t, j} = 0.
\end{equation}
where $c_{t, j}$ is the normalizing constant satisfying $\int_X e^{F_{t, j} + c_{t, j}} \Omega_t = [\theta_t]^n$. Obviously, for each fixed $j$, $c_{t, j}$ is uniformly bounded for all $t\in \D$.   Furthermore, for fixed $j>0$, $\limsup_{j\rightarrow \infty} \lim_{t\rightarrow 0} c_{t, j} = 0$ since 
$\int_{\mathcal{X}_0} e^F \Omega_0 = [\theta_0]^n$.  

\v

Obviously there exists a unique bounded solution $\phi_{t, j}$ to equation (\ref{appma}) for all $t\in \D$. In particular, $\phi_{t, j}$ is smooth if $t\neq 0$.  We will obtain some uniform estimates for $\phi_{t, j}$.

\begin{lemma} \label{lpbd2}  There exists $C>0$ such that for  any $j>0$, there exists $\delta_j >0$ such that for all $t\in \D$ with $|t|<\delta_j$, we have   
$$\left\|  e^{F_{t, j}+c_{t, j}}  \right\|_{L^p(\mathcal{X}_t, \Omega_t) }\leq C. $$

\end{lemma}

\begin{proof} We first notice that $F_{0, j} \leq F$ on $X$ and so $\| F_{0, j} \|_{L^p(X, \Omega_0)}$ is uniformly bounded. For each fixed $j$, $F_{t, j}$ is uniformly bounded above for all $t\in \D$ and it converges smoothly to $F_{0, j}$ on $X^{\rm reg}$. Since the singular set of $X$ has measure $0$ with respect to $\Omega_0$, we obtain uniform $L^p$ bound  for $F_{t, j}$ for all $t\in \D$ sufficiently close to $t=0$. The estimate follows as for each fixed $j$, $|c_{t, j}|$ is uniformly bounded for sufficiently small $|t|$. 
\end{proof}

\v

\begin{lemma} \label{lpbd3} There exist $q>1$ and $C>0$ such that for  any $j>0$, there exists $\delta_j >0$ such that for all $t\in \D$ with $|t|<\delta_j$, we have   
$$\left\| \frac{ e^{F_{t, j} + c_{t, j}}\Omega_t}{\theta_t^n} \right\|_{L^q(\mathcal{X}_t, \theta_t^n) }\leq C. $$

\end{lemma}

\begin{proof} The lemma immediately follows from Lemma \ref{lpbd1} and Lemma \ref{lpbd2} combined with H\"older inequality. 
\end{proof}

\v

\begin{lemma}  \label{linfestj} There exists $C>0$ such that for $j>0$, there exists $\delta_j>0$ such that for all $t\in \D$ with $|t|<\delta_j$ we have 
$$\| \phi_{t, j} \|_{L^\infty(\mathcal{X}_t)} \leq C. $$

\end{lemma}

\begin{proof} The lemma is proved in \cite{GGZ1}. The key ingredient of the proof is to obtain the uniform bound on Tian's $\alpha$-invariant on $(\mathcal{X}_t, \theta_t)$ established in \cite{GGZ1, LY2} (c.f Theorem A \cite{GGZ1}).  Then one can proceed with the same argument in \cite{K, EGZ} to achieve the $L^\infty$ estimate due to the uniform $L^p$ bound on the volume measure from Lemma \ref{lpbd3}. 
\end{proof} 

\v

Let $\omega_{t, j} = \theta_t + \ddb \phi_{t, j}$. Then the Ricci curvature of $\omega_{t, j}$ is given by 
\begin{equation}\label{riclow1}
Ric(\omega_{t, j}) = Ric(\Omega_t) - \ddb F_{t, j} \geq - (A+B) \theta_t.
\end{equation}

The following Schwarz lemma is achieved by the same technique as in \cite{ST1}.

\begin{lemma}  There exists $\delta>0$ such that for all $j>0$, there exists $\delta_j>0$ such that  all $t\in \D\setminus\{0\}$ with $|t|<\delta_j$,  we have 
\begin{equation}\label{schtj}
\omega_{t, j} \geq  \delta  \theta_t
\end{equation}
on $\mathcal{X}_t$,

\end{lemma}

\begin{proof} We consider the quantity $$H_{t, j} = \log tr_{\omega_{t, j}}(\theta_t) - K \phi_{t, j}$$ for a fixed sufficiently large $K>0$. By the Ricci curvature estimate (\ref{riclow1}) for $\omega_{t, j}$ and the uniform $L^\infty$-estimate for $\phi_{t, j}$ in Lemma \ref{linfestj}, we can immediately apply the Chern-Lu type estimate by the same calculation in \cite{ST1, ST2} to  achieve (\ref{schtj}). 
\end{proof}

\v

We immediately obtain a uniform lower Ricci bound for $\omega_{t, j}$.

\begin{corollary} There exist $\Lambda>0$ and $D$ such that for  all $j>0$, there exists $\delta_j>0$ such that  all $t\in \D\setminus\{0\}$ with $|t|<\delta_j$,  we have
\begin{equation}\label{riclow2}
Ric(\omega_{t, j}) \geq  - \Lambda ~\omega_{t,j} 
\end{equation}
on $\mathcal{X}_t$ and 
\begin{equation}\label{dia1}
diam(\mathcal{X}_t, \omega_{t, j}) \leq D.
\end{equation}

\end{corollary} 

\begin{proof} Direct calculations show that 
$$Ric(\omega_{t, j}) = Ric(\Omega_t) - \ddb F_{t, j} \geq - (A+B) \theta_t \geq  - \delta^{-1} (A+B) \omega_{t, j}$$
by (\ref{schtj}).
The Ricci lower bound is then proved by letting $\Lambda = \delta^{-1} (A + B)$. The diameter bound immediately follows from the general geometric estimates established in \cite{GPSS1} (c.f. \cite{GPSS1}) due to the uniform $L^p$ bound of Lemma \ref{lpbd3} for the volume measure $(\omega_{t, j})^n$. 
\end{proof}

\v

The above Schwarz lemma gives local 2nd order estimates for $\phi_{t, j}$ away from the singular set of $X$. In particular, for fixed $j>0$, one can obtain the local $C^\infty$ convergence of $\phi_{t, j}$ to $\phi_{0, j}$ on $X^{\rm reg}$. Hence $\phi_{0, j}$ is smooth on $C^\infty(X^{\rm reg})$.

\v

 Let $L_t$ be the ample line bundle with $c_1(L_t) = [\theta_t]$. We can choose the hermtian metric $h_{t, j}$ on $L_t$ with 
$$Ric(h_{t, j}) = \omega_{t, j}.$$
We can further obtain the following global Gromov-Hausdorff convergence.

\begin{lemma} \label{gwcon} For each $j>0$,  $(\mathcal{X}_t, \omega_{t, j})$ converge in Gromov-Hausdorff topology as $t\rightarrow 0$ to a compact metric space $(Y_j, d_j)$ satisfying
\begin{enumerate}

\item $(Y_j, d_j)$ coincides with the metric completion of $(X^{\rm reg}, \omega_{0, j})$,

\medskip

\item $(Y_j, d_j)$ is homeomorphic to projective variety $X$ itself. 

\medskip

\item $\omega_{t, j}$ converges smoothly to $\omega_{0, j}$ on $X^{\rm reg}$.

\end{enumerate}

\end{lemma}

\begin{proof}  By the uniform Ricci lower bound and diameter bounds for $(\mathcal{X}_{t, j}, \omega_{t, j})$ for $|t|< \delta_j$, one can immediately obtain uniform partial $C^0$-estimate by \cite{DS, LZ} for $(\mathcal{X}_{t, j}, \omega_{t, j}), kL_t, h^k_{t, j})$ for sufficiently large $k\in \mathbb{Z}^+$.  The lemma follows from the fact that $\phi_{t, j}$ converges smoothly to $\phi_{0, j}$. 
\end{proof}

\v

\begin{lemma}  \label{apphol} There exist $\gamma>0$ and $C>0$ such that for any $j>0$, there exists $\delta_j>0$ such that for any $t\in \D\setminus \{0\}$ with $|t|< \delta_j$, we have
\begin{equation}\label{holtj}
|\phi_{t, j}(x) - \phi_{t, j}(y)| \leq C d_{\theta_t}(x, y)^\gamma 
\end{equation}
for all $x, y \in \mathcal{X}_t$.

\end{lemma}

\begin{proof} By Lemma \ref{gwcon}, the partial $C^0$ estimate holds for all $(\mathcal{X}_t, \omega_{t, j})$ with $t\in \D\setminus \{0\}$ and $|t|<\delta_j$. %
Then there exist $k>0$ such that the orthonormal basis of $H^0(\mathcal{X}_t, kL_t)$ with respect to $h_{t, j}^k$ and $k \omega_{t, j}$ induces a projective embedding 
 $$T_{t, j}: \mathcal{X}_t \rightarrow \mathbb{P}^M $$
 for a fixed projective ambient space $\mathbb{P}^M$. We can assume $\mathcal{X}$ is also embedded in $\mathbb{P}^M$ after embedding $\mathbb{P}^N$ into $\mathbb{P}^M$ by choosing a sufficiently large $k$. Let $\Theta_{\mathbb{P}^M}$ be the Fubini-Study metric on $\mathbb{P}^M$ and 
% %
 $$\omega_{FS, t, j} = \frac{1}{k} T_{t, j}^*\Theta_{\mathbb{P}^M}. $$ 
 One can also rewrite 
 $$\omega_{t, j} = \omega_{FS, t, j} + \ddb \psi_{t, j}, ~\sup_{\mathcal{X}_t} \psi_{t, j} = 0.$$

 \v   We  immediately apply Theorem \ref{m} to obtain the H\"older estimate for $\varphi_{t, j}$, i.e., there  exist $\gamma>0$ and $C>0$ such that for any $j>0$, there exists $\delta_j>0$ such that for any $t\in \D\setminus \{0\}$ with $|t|< \delta_j$, we have
\begin{equation}\label{holtj}
|\psi_{t, j}(x) - \psi_{t, j}(y)| \leq C d_{\theta_t}(x, y)^\gamma 
\end{equation}
for all $x, y \in \mathcal{X}_t$.

\v

On the other hand,   we can extend $$T_{t, j}: \mathcal{X}_t \subset \mathbb{P}^N \ra \mathbb{P}^M$$ to 
$$\tilde T_{t, j} \in SL(M+1, \mathbb{C}): \mathbb{P}^M \rightarrow \mathbb{P}^M. $$ Since $\phi_{t, j}$ and $\psi_{t, j}$ are both  uniformly bounded in $L^\infty(\mathcal{X}_t)$ for all $j$ and $|t|<\delta_j$, $\tilde T_{t, j}$ must lie in a bounded subset set of $SL(M+1, \mathbb{C})$ for all $j>0$ and $|t|<\delta_j$. In particular, this implies that $\theta_t$ and $\omega_{FS, t, j}$ are uniformly equivalent for all $j>0$ and $t\in \D \setminus\{0\}$ with $|t|<\delta_j$. This completes the proof of the lemma. 
\end{proof}

\v

Finally we are able to prove Theorem \ref{mainthm3}.

\begin{proposition} There exists $\gamma>0$ and $C>0$ such that for all $x, y \in X$, we have 
$$|\phi(x) - \phi(y)| \leq C d_{\mathbb{P}^N}(x, y)^\gamma. $$

\end{proposition}

\begin{proof} By the uniform H\"older estimate in Lemma \ref{apphol}, for each $j>0$, we can pass the H\"older estimate to the limiting space $X$. More precisely, there exists $\gamma>0$ and $C>0$ independent of $j>0$, such that for all $x, y\in X$ we have 
\begin{equation}\label{hollim}
|\phi_{0, j} (x) - \phi_{0, j}(y)| \leq C d_{\mathbb{P}^N}(x, y)^\gamma.
\end{equation}
Since $\lim_{j\rightarrow \infty} \left\|e^{F_{0, j}+c_{0, j}} - e^F \right\|_{L^1(X, \textcolor{red}{\Omega^t})}=0$, the stability theorem for degenerate complex Monge-Amp\`ere equation of \cite{DZ} implies that 
\begin{equation}\label{stabl}
\lim_{j\rightarrow \infty} \left\| \phi_{0, j} - \phi \right\|_{L^\infty(X)} = 0.
\end{equation}
The proposition then immediately follows from (\ref{hollim}) and (\ref{stabl}). 
\end{proof}

\v
\v
\v

\noindent {\bf{Acknowledgements:}} This project was initiated  during the AIM workshop on `PDE methods in complex geometry' in August, 2024. The authors would like to thank the generous support and the hospitality of American Institute of Mathematics. 
\thanks{This work is supported in part by the National Science Foundation under grants DMS-2203607 and DMS-2303508.}

\v
\v
\v
\setcounter{equation}{0}


\begin{thebibliography}{99}
{\small

\bibitem{A}
Aubin, T., {\em Equations du type Monge-Amp\`ere sur les vari\'et\'es K\"ahl\'eriennes compactes},
Bull. Sci. Math. (2) 102 (1978), no. 1, p. 63–95.
\v
\bibitem{BandoM} Bando, S., Mabuchi, T., {\em Uniqueness of Einstein K\"ahler metrics modulo connected group actions}.
Algebraic Geometry, Sendai 1985, Adv. Stud. Pure Math. 10 (1987), 11--40.



\v
\bibitem{BG}
Berman, R. Guenancia, H.,
{\em K\"ahler-Einstein metrics on stable varieties and log canonical pairs}
Geom. Funct. Anal. 24 (2014), no. 6, 1683–1730.
\v


\bibitem{BM}
Bierstone, E., Milman, P.,
{\em Semianalytic and subanalytic sets},
Inst. Hautes Études Sci. Publ. Math. No. 67 (1988), 5--42.

\v
\bibitem{BM2000}Bierstone, E., Milman, P.,
{\em Subanalytic Geometry},
Model Theory, Algebra, and Geometry
MSRI Publications
Volume 39, 2000



\v
\bibitem{Cal}
Calabi,E., {\em On K\"ahler manifolds with vanishing canonical class}, Algebraic Geometry and Topology. A Symposium in Honor of S. Lefschetz, Princeton University Press, 1957, pp. 78--89.
\v
\bibitem{CDS1} Chen, X., Donaldson, S., Sun, S.,
{\em Kähler-Einstein metrics on Fano manifolds, I: Approximation of metrics with cone singularities}, J. Amer. Math. Soc. 28 (2015), no. 1, 183–197. 


\v
\bibitem{CDS2} Chen, X., Donaldson, S., Sun, S.,
{\em Kähler-Einstein metrics on Fano manifolds, II: Limits with cone angle less than $2\pi$}, J. Amer. Math. Soc. 28 (2015), no. 1, 199-234


\v
\bibitem{CDS3} Chen, X., Donaldson, S., Sun, S.,
{\em Kähler-Einstein metrics on Fano manifolds, III: Limits as cone angle approaches $2\pi$ and completion of the main proof}, J. Amer. Math. Soc. 28 (2015), no. 1, 235–278

\v
\bibitem{CG} Coman, D., Guedj, V. and Zeriahi, A.  {\em On the extension of quasiplurisubharmonic functions},  Anal. Math. 48 (2022), no. 2, 411--426


\v
\bibitem{C}
Croke, C. {\em Some isoperimetric inequalities and eigenvalue estimates}, Ann. Sci. Ecole Norm.
Sup. (4) 13 (1980), 419--435.

\v

\bibitem{DZ} Dinew, S. and Zhang, Z. {\em On stability and continuity of bounded solutions of degenerate complex Monge-Amp\`ere equations over compact K\"ahler manifolds}, 
Adv. Math.   225 (2010), no. 1, 367--388

\v

\bibitem{DGG} Di Nezza, E., Guedj, V. and Guenancia, H. {\em Families of singular K\"ahler-Einstein metrics}, J. Eur. Math. Soc. 25 (2023), no. 7, pp. 2697--2762

\v

\bibitem{Don}
Donaldson, S. K.,
{\em Scalar curvature and projective embeddings. I}, 
J. Differential Geom. 59 (2001), no. 3, 479–522.

\v


\bibitem{DS}
Donaldson, S., Sun, S.,
{\em Gromov-Hausdorff limits of Kähler manifolds and algebraic geometry},
Acta Math. 213 (2014), no. 1, 63–106.
\v

\bibitem{EGZ}  Eyssidieux, P., Guedj, V., Zeriahi, A., {\em Singular Kähler-Einstein metrics}
J. Amer. Math. Soc. 22 (2009), no. 3, 607–639.
 {\em }

\bibitem{FGS1}  Fu, X., Guo, B. and  Song, J. {\em Geometric estimates for complex Monge-Amp\`ere equations}, J. Reine Angew. Math. 765 (2020), 69--99

\bibitem{FGS2} Fu, X., Guo, B. and  Song, J. {\em  RCD structures on singular K\"ahler spaces of complex dimension three}, arXiv:2503.08865
 
 
\v
\bibitem{GGZ1} Guedj, V., Guenancia, H. and Zeriahi, 
A. {\em Continuity of singular K\"ahler-Einstein potentials.}, 
Int. Math. Res. Not. IMRN 2023, no. 2, 1355–1377.
 \v
 
\bibitem{GGZ2} Guedj, V., Guenancia, H. and Zeriahi, 
A.  {\em Strict positivity of K\"ahler-Einstein currents},  Forum Math. Sigma 12 (2024), Paper No. e68
 
 \v
 
\bibitem{GGZd} Guedj, V., Guenancia, H. and Zeriahi, A.,
{\em Diameter of K\"ahler currents},
J. Reine Angew. Math. 820 (2025), 115--152.
 \v
 
 
 \bibitem{GPSS1} Guo, B., Phong, D.H., Song, J.  and Sturm, J. {\em Sobolev inequalities on K\"ahler spaces}, 2023, arXiv:2311.00221

\v

\bibitem{GPSS2}Guo, B., Phong, D.H., Song, J.  and Sturm, J.  {\em Diameter estimates in K\"ahler geometry},  {Comm. Pure Appl. Math.}, Volume 77, Issue 8 (2024), 3520--3556

\v


\bibitem{GPSS3}Guo, B., Phong, D.H., Song, J.  and Sturm, J.   {\em Diameter estimates in K\"ahler geometry II: removing the small degeneracy assumption}, Math. Z. 308, 43 (2024) 

\v


\bibitem{GPTW} Guo, B., Phong. D., Tong, F., Wang, C.,
{\em On the modulus of continuity of solutions to complex Monge-Ampère equations},
arXiv:2112.02354

\v

\bibitem{GS25} Guo, B. and Song, J. {\em Nash entropy, Calabi energy and geometric regularization of singular K\"ahler metrics}, arXiv:2502.02041



\v

\bibitem{Kob} Kobayashi, R., {\em Einstein-K\"ahler V-metrics on open Satake V-surfaces with isolated quotient singularities}. Mathematische Annalen 272 (1985), 385--398.
 
 
 \v
 \bibitem{K} Ko\l dziej, S.,
The complex Monge-Amp\`ere equation.
Acta Math. 180 (1998), no. 1, 69–117.
 
 \v
 
\bibitem{K1} Ko\l odziej, S.,
{\em H\"older continuity of solutions to the complex Monge-Amp\'ere equation with the right-hand side in $L^p$: the case of compact K\"ahler manifolds. },
Math. Ann. 342 (2008), no. 2, 379–386.

 
 %\v \bibitem{KS} Koll\'ar, Shepherd-Barron N., {\em Threefolds and deformations of surface
%singularities}, Invent. Math. 91 (1988), no. 2, 299–338.

\v

 \bibitem{L} Li, C., {\em K\"ahler-Einstein metrics and K-stability}, Thesis, available on author's web page, https://sites.math.rutgers.edu/~cl1412/
 
\v
\bibitem{LG}
Li, C.,
{\em G-uniform stability and K\"ahler-Einstein metrics on Fano varieties}, Invent. Math. 227 (2022), no. 2, 661--744
\v


\bibitem{LY1} Li, Y. {\em On collapsing Calabi-Yau fibrations}, J. Differential Geom., 117(3), (2021), 451--483 

\bibitem{LY2} Li, Y. {\em Uniform Skoda integrability and Calabi-Yau degeneration}, Anal. PDE 17 (2024), no. 7, 2247--2256

 

%\v
\bibitem{LWX}
Li, C., Wang, X., Xu, C., {\em
On the proper moduli spaces of smoothable K\"ahler-Einstein Fano varieties.}
Duke Math. J. 168 (2019), no. 8, 1387–1459.


\v
\bibitem{LXZ}
Liu, Y., Xu, C., Zhuang Z., {\em
Finite generation for valuations computing stability thresholds and applications to K-stability.} Ann. of Math. (2) 196 (2022), no. 2, 507–566
\v

\bibitem{LZ} 

Liu, G., Sz\'ekelyhidi, G.,
{\em Gromov-Hausdorff limits of Kähler manifolds with Ricci curvature bounded below,}
Geom. Funct. Anal. 32 (2022), no. 2, 236--279.
\v
\bibitem{Lo}
Lojasiewicz, S., {\em Ensemble semi-analytiques}, Preprint IHES, 1965.
\v
\bibitem{PS1}
Phong, D. H., Sturm, J.,
{\em The Monge-Ampère operator and geodesics in the space of Kähler potentials.}
Invent. Math. 166 (2006), no. 1, 125–149.
\v
\bibitem{PS2}
Phong, Duong H.,Sturm, J.,
{\em Test configurations for K-stability and geodesic rays.} 
J. Symplectic Geom. 5 (2007), no. 2, 221–247.
\v

\bibitem{Siu}
Siu, Y.T., {\em Techniques for the analytic proof of the finite generation of the canonical ring}, arXiv:0811.1211.


\v
\bibitem{SSY}
Spotti, C., Sun, S., Yao, C.,
{\em Existence and deformations of K\"ahler-Einstein metrics on smoothable Q-Fano varieties.} 
Duke Math. J. 165 (2016), no. 16, 3043--3083.

%\v
%\bibitem{M}
%Mumford, D., {\em Stability of projective varieties} L’Enseignement Math\'ematique, Geneva,
%1977. Lectures given at the “Institut des Hautes Etudes Scientifiques”, Bures-sur-Yvette, ´
%March-April 1976; Monographie de l’Enseignement Math´ematique, No. 24
\v

\bibitem{Sg}
Sugiyama, K., {\em Einstein-K\"ahler Metrics on Minimal Varieties of General Type}, Recent Topics in Differential and Analytic Geometry, T. Ochiai, ed., Adv. Stud. in Pure Math. 18-I (1990),


\v
\bibitem{Ts} Tsuji,H., {\em Existence and degeneration of K\"ahler-Einstein metrics on minimal algebraic varieties of general type}. Math. Ann. 281 (1988), no. 1, 123–133


\v 


\bibitem{Song1} Song, J., {\em The $\alpha$-invariant on $\mathbb{CP}^2$ blown up at two points}, Trans. Amer. Math. Soc. 357 (2005), no. 1, 45--57 

\v

\bibitem{Song2} Song, J., {\em The $\alpha$-invariant on Toric Fano Manifolds}, Amer. J. Math. 127 (2005), no. 6, 1247--1259

\v
\bibitem{Song3} Song, J., {\em Riemannian geometry of K\"ahler-Einstein currents}
arXiv:1404.0445

\v

 

\bibitem{SSW} Song, J., Sturm, J. and Wang, X. {\em Riemannian geometry of K\"ahler-Einstein currents III: compactness of K\"ahler-Einstein manifolds of negative scalar curvature}, arXiv:2003.04709


\bibitem{ST1} Song, J. and Tian, G. {\em The K\"ahler-Ricci flow on surfaces of positive Kodaira dimension}, Invent. Math. {\bf 170} (2007), no. 3, 609--653

\v

 \bibitem{ST2} Song, J. and Tian, G. {\em Canonical measures and K\"ahler-Ricci flow}, J . Amer. Math. Soc. 25 (2012), 303--353

\v

 \bibitem{ST3} Song, J. and Tian, G. {\em The K\"ahler-Ricci flow through singularities}, with G. Tian, Invent. Math. 207 (2017), no. 2, 519--595


\v

\bibitem{Sz1}  Sz\'ekelyhidi, G. {\em The partial C0-estimate along the continuity method}, J. Amer. Math. Soc. 29 (2016), 537--560

\v

\bibitem{Sz2}  Sz\'ekelyhidi, G. {\em Singular K\"ahler-Einstein metrics and RCD spaces},   arXiv:2408.10747

\v

\bibitem{SzT}  Sz\'ekelyhidi, G. and Tosatti, V.  {\em Regularity of weak solutions of a complex Monge-Amp\`ere equation},  Analysis $\&$ PDE 4 (2011), n. 3, 369--378

\v


 \bibitem{T1} Tian, G. {\em On Calabi's conjecture for complex surfaces with positive first Chern class},  Invent. Math. 101 (1990), no. 1, 101--172

\v
 \bibitem{T2} Tian, G. {\em  K\"ahler-Einstein metrics with positive scalar curvature}, Invent. Math. 130 (1997),
no.1, 1--37



\v 

\bibitem{T3} Tian, G. {\em K-stability and K\"ahler-Einstein metrics},  Comm. Pure Appl. Math. 68 (2015), no. 7, 1085--1156




\v
\bibitem{Ts} Tsuji, H.,  {\em Existence and degeneration of K\"ahler-Einstein metrics on minimal algebraic varieties of general type}, Math. Ann. 281 (1988), no. 1, 123--133
 

\v
\bibitem{Y} Yau, S.T.,
{\em On the Ricci curvature of a compact K\"ahler manifold and the complex MongeAmp\`ere equation. I}, Commun. Pure Appl. Math. 31 (1978), p. 339–411.

\v
\bibitem{Z}  Zhang, K.,
 {\em Some refinements of the partial $C^0$ estimate,}
Anal. PDE 14 (2021), no. 7, 2307--2326.
}
\end{thebibliography}
\end{document}